\documentclass[reqno,11pt]{amsart}
\usepackage[margin=1in]{geometry}

\usepackage{amsmath,amssymb,amsthm,enumerate}
\usepackage[dvipsnames]{xcolor}
\usepackage{mathtools}
\usepackage[font=small,labelfont=bf]{caption}
\usepackage[subrefformat=parens]{subcaption}
\usepackage[colorlinks, linkcolor=blue, citecolor=magenta, hypertexnames=false,backref]{hyperref} 
\numberwithin{equation}{section}

\usepackage{graphicx,tikz}
\usepackage{bbm}
\usetikzlibrary[arrows,arrows.meta,decorations.pathreplacing]
\newtheorem{theorem}{Theorem}[section]

\newtheorem{proposition}{Proposition}[section]

\newtheorem{conjecture}{Conjecture}[section]
\newtheorem{lemma}{Lemma}[section]

\theoremstyle{definition}

\theoremstyle{remark}
\newtheorem{remark}{Remark}[section]

\DeclareMathOperator{\Tr}{Tr}
\newcommand{\overbar}[1]{\mkern 1.5mu\overline{\mkern-1.5mu#1\mkern-1.5mu}\mkern 1.5mu}

\def \RP {\mathbb{RP}}

\newcommand\numberthis{\addtocounter{equation}{1}\tag{\theequation}}

\begin{document}

\title[]{Geodesic distance Riesz energy on  projective spaces}

\author[]{Dmitriy Bilyk, Ryan W. Matzke, Joel Nathe}
\address{School of Mathematics, University of Minnesota, Minneapolis, MN 55455, USA}
\email{dbilyk@umn.edu}

\address{Department of Mathematics, Vanderbilt University, Nashville, TN 37240, USA}
\email{ryan.w.matzke@vanderbilt.edu}

\address{School of Mathematics, University of Minnesota, Minneapolis, MN 55455, USA}
\email{nathe031@umn.edu}

\begin{abstract}
We study probability measures that minimize the Riesz energy  with respect to the geodesic distance $\vartheta (x,y)$ on  projective spaces  $\mathbb{FP}^d$ (such energies arise from the 1959 conjecture of Fejes T\'oth about sums of non-obtuse angles), i.e. the integral $$ \frac{1}{s} \int_{\mathbb{FP}^d} \int_{\mathbb{FP}^d} \big( \vartheta (x,y) \big)^{-s} d\mu(x) d\mu (y) \,\,\, \textup{ for } \,\,\, s<d$$
and find ranges of the parameter $s$ for which the energy is minimized by the uniform measure $\sigma$ on $\mathbb{FP}^d$. To this end, we use  various methods of harmonic analysis, such as Ces\`aro  averages of Jacobi expansions and $A_1$ inequalities, and establish a rather general theorem guaranteeing that certain energies with singular kernels are minimized by $\sigma$.

In addition, we obtain further results and present  numerical evidence, which uncover a peculiar effect  that  minimizers this energy   undergo numerous phase transitions, in sharp contrast with many analogous known examples (even the seemingly similar geodesic Riesz energy on the sphere), which usually have only one transition (between uniform and discrete  minimizers). 

\end{abstract}

\maketitle

\section{Introduction and background}\label{s.intro}

The starting point of this paper is the following energy minimization problem. For two points on the sphere $x,y \in \mathbb S^{d} \subset \mathbb R^{d+1}$, let us define 
\begin{equation}\label{e.theta}
\theta (x, y) = \min \{ {\arccos  (x\cdot y) , \pi - \arccos  (x\cdot y) \} = \arccos  (| x\cdot y|)}, 
\end{equation}
i.e. $\theta (x, y)$ is the acute (non-obtuse) angle between the lines generated by the vectors $x$ and $y$. For $\lambda \in \mathbb R$,  our goal is  to understand the optimizers (minimizers for $\lambda \le 0$ and maximizers for $\lambda >0$) of the energy integral 
 \begin{equation}\label{e.def}
 I_\lambda (\mu) =  \int_{\mathbb S^d} \int_{\mathbb S^d} \big( \theta (x,y)  \big)^\lambda \, d\mu (x) \, d\mu (y)
 \end{equation}
 over measures $\mu$ in the set $\mathcal P (\mathbb S^d)$ of Borel probability measures on $\mathbb S^d$ (for $\lambda = 0$ the kernel is replaced by $-\log \theta (x,y)$). 
 
 The motivation for studying this question is two-fold. First, it is directly connected to 
 a conjecture about the sum of acute angles between lines posed by Fejes T\'oth in 1959 \cite{FejT}, see \S\ref{s.FT}, and has been studied in this context \cite{BilM,LimM21,LimM22}. Second,  perhaps even more importantly, since $\theta (x,y)$ can be interpreted as the geodesic distance on the real projective space (see \S\ref{s.proj} for more details), the energy integral \eqref{e.def} falls into the general category of Riesz energies (distance integrals) which are ubiquitous in potential theory, mathematical physics, metric geometry, etc. 
 
Under this interpretation, we uncover a very unusual phenomenon: optimizers of $I_\lambda (\mu)$  exhibit behavior which is quite different from  other known examples of analogous distance integrals, even the seemingly similar geodesic distance Riesz energy on the sphere $\mathbb S^d$. In particular, rather than one simple phase transition between continuous and discrete optimal measures (see \S\ref{s.typical} for typical  examples), optimizers of $I_\lambda (\mu)$  appear to undergo a cascade of phase transitions with optimal measures of different dimensions.

 We study the energy $I_\lambda$ in the  broader context of geodesic distance Riesz energies on compact connected two-point homogeneous spaces, and the main results about optimizers of $I_\lambda$, Theorems \ref{t.main} ($d\ge 2$) and \ref{t.d1} ($d=1$), are consequences of the much more general Theorem \ref{t.geodriesz} about such energies.  To prove this theorem, in \S\ref{sec:Jacobi Expansion of Geod. Riesz} we present a rather technical computation establishing positivity of Jacobi coefficients of the geodesic distance Riesz kernels for a certain range of exponents, and in \S\ref{s.singularkernels} we use and develop a variety of analytic tools to finish the proof, as well as obtain the first rather general result, Theorem \ref{thm:general}, about minimization of energies with singular kernels by the uniform measure. In \S\ref{s.d1} we further discuss the case $d=1$ and its relation to the geodesic Riesz energy on the sphere. 
 
 In addition, in \S\ref{s.<1} we present results and numerical evidence, which demonstrate surprising behavior of maximizers of $I_\lambda$ on $\mathbb S^d$ for $0<\lambda <1$. Section \ref{sec.discrete} turns to the case $\lambda\ge 1$ and presents  results about discreteness of  optimal measures, which are based on the characterization of optimizers for the chordal Riesz energy in \S\ref{s.chordal}.    Below we present the motivation, background, and the main results of the paper in more detail. 


\subsection{Fejes T\'oth conjecture on the sum of acute angles} \label{s.FT}
In 1959  Fejes T\'oth  \cite{FejT}  posed the following natural conjecture in discrete geometry:

\begin{conjecture}[{\bf{Fejes T\'oth conjecture}}]\label{c.FT} Let   $x_1$, $x_2,\dots, x_N$ be (not necessarily distinct) points on the sphere $\mathbb S^{d}$. Then the discrete energy
\begin{equation}
\sum_{i,j=1}^N \theta (x_i, x_j)
\end{equation}
is maximal when the points $x_1,\dots, x_{d+1}$ are mutually orthogonal and $x_k = x_{k-(d+1)}$ for all $k$ with $d+1<k\le N$.
\end{conjecture}
 In plain words, the conjecture states that  the sum of non-obtuse angles between $N$ points on the sphere is maximized by  periodically repeated elements of an orthonormal basis. In particular, if $N=k(d+1)$, then $k$ copies of the same orthonormal basis of $\mathbb R^{d+1}$ should maximize the sum. Despite its simple formulation, the conjecture is still open for $d>1$ (even in the aforementioned special case), see \cite{BilM,FodVZ,Gor,Pet} for recent partial results.    

Fejes T\'oth stated the conjecture on $\mathbb S^2$ (and proved it for $N\le 6$), but the generalization of the conjecture to other dimensions $d\ge 1$ is straightforward.
This conjecture admits a natural continuous {analogue}:

\begin{conjecture}[{\bf{Integral Fejes T\'oth conjecture}}]\label{c.IFT} The maximal value of the energy integral
\begin{equation}\label{e.ift}
I_1 (\mu) = \int_{\mathbb S^d} \int_{\mathbb S^d} \theta (x,y) \, d\mu (x) \, d\mu (y)
\end{equation} 
over all Borel probability measures on $\mathbb S^d$ is achieved by the  uniform measure on the elements of an orthonormal basis $\{e_i\}_{i=1}^{d+1}$ of $\mathbb R^{d+1}$, i.e. $\displaystyle{\mu_{ONB} = \frac{1}{d+1} \sum_{i=1}^{d+1} \delta_{e_i}}$. 
\end{conjecture}

The continuous conjecture is equivalent to the discrete one in the case when $N$ is a multiple of the ambient dimension, i.e. $N= k (d+1)$. 
 Just as in the discrete case, it remains open,  except for the case $d=1$, i.e. on the circle $\mathbb S^1$ 
  (this case  will be revisited in Section \ref{s.d1}).  Various versions and modifications of these conjectures have attracted a fair amount  of attention in the recent years from several groups of authors \cite{BilM,FodVZ,Gor,LimM21,LimM22}. 
  
  Some  of this work dealt with optimizing the energy with the kernel given by powers of the acute angle $\theta (x,y)$, i.e. precisely the integral $I_\lambda (\mu)$.   In particular, it was proved in \cite{LimM22} that there exists $\lambda_0 \in [1,2)$ such that for all $\lambda \ge \lambda_0$, a version of the Fejes T\'oth conjecture holds for $I_{\lambda}$, i.e. the energy $I_\lambda  (\mu)$ is maximized by the measure $\displaystyle{\mu_{ONB}}$. 
  Notice that Conjecture \ref{c.IFT} essentially states that $\lambda_0 =1$, see \S\ref{sec.discrete} for more details.  In the present paper we are mostly interested in the behavior of the energy optimizers of  $I_\lambda (\mu) $ for $\lambda <1$, and we shall see that this behavior  is quite atypical. 

 \subsection{Distance integrals (Riesz energies) and the typical behavior of optimizers}\label{s.typical}

 Observe that the kernel $\theta (x,y)$ remains invariant if  $x\in \mathbb S^d$ is replaced by its  opposite {(antipode)} $-x$. Therefore, it is natural to view the question as a problem on the real projective space $\mathbb{RP}^d$. In fact, in this case, up to a multiplicative constant, $\theta (x,y)$ is the geodesic distance  on $\mathbb{RP}^d$, see \eqref{e.rhotheta}. 

Therefore, the main object of our study, the energy integral $I_\lambda (\mu)$  defined in  \eqref{e.def} is an example of a distance integral --  a common object in metric geometry. 
Let $(X,\rho)$ be a compact metric space, for $\lambda \in \mathbb R$, and a Borel probability measure $\mu$ on $X$, the distance integral is defined as  
\begin{equation}\label{e.di}
 I_{\rho,\lambda} (\mu) =  \int_{ X} \int_{ X} \big( \rho (x,y)  \big)^\lambda \, d\mu (x) \, d\mu (y),
 \end{equation}
where for  $\lambda =0$ one employs the logarithmic kernel $- \log \rho (x,y)$. Maximizers of these energies (or minimizers when  $\lambda \le 0$) provide information about the geometric structure of the metric space $X$. This optimization is a continuous analogue of optimizing sums of powers of pairwise distances between $N$ points in $X$. In the case when $X\subset \mathbb R^n$  and $\rho (x,y) = \| x-y \|$, these objects are commonly known as Riesz energies, and this name often propagates to other metric spaces. 

The behavior of optimizers of distance integrals in standard examples typically goes through  one  phase transition: for large $\lambda$, above a certain threshold $\lambda_0$ depending on the space, the maximizers are discrete measures (often supported on sets of points separated by $\operatorname{diam}(X)$, while below the threshold  optimizers are often unique and are given by an isometry-invariant measure, if it exists, e.g. the uniform measure $\sigma$ on $\mathbb S^d$ or $\mathbb{RP}^d$ (or other two-point homogeneous spaces) or the Haar measure if $X$ admits group structure.  
{At the transition value $\lambda = \lambda_0$, the class of optimizers is generally fairly rich and contains measures supported on sets of different dimensions.

Let us provide some examples to illustrate the above heuristic.
\begin{itemize}
\item compact $X\subset  \mathbb R^n$, $\rho (x,y) = \| x-y \|$:  for $\lambda >2$, the maximizers are discrete with at most $n+1$  extremal points of {the convex hull of} $X$ in the support \cite[Theorem 12]{Bj}, while for each $\lambda \in (-\operatorname{dim} (X),  2)$ there is a unique  optimizer, see e.g.
 \cite[Theorem 4]{Bj}, \cite[Theorems 4.4.5 and 4.4.8]{BorHS}.  
\item $X = \mathbb S^d$,  $\rho (x,y) = \| x-y \|$: for $\lambda > 2$ the maximizers are discrete measures with mass $\frac12$ at two opposite poles, i.e. measures of the form $\mu = \frac12 (\delta_x + \delta_{-x})$, while for $-d < \lambda < 2$, the unique optimizer is the uniform surface measure $\sigma$; when $\lambda =2$, maximizers are all measures with the center of mass at the origin \cite[Theorems 5 and 7]{Bj}, \cite[Theorems 4.6.4 and 4.6.5]{BorHS}. 
\item   $X = \mathbb S^d$ and  $\rho$ is the geodesic distance: in this case the phase transition is different ($\lambda_0 =1$) with the same optimizers as in the previous example for $\lambda > \lambda_0$ and $-d  < \lambda < \lambda_0$, while at $\lambda = \lambda_0 =1$ maximizers are exactly all centrally symmetric measures \cite[Theorem 1.1]{BilD}.
\item  $X = \{-1,1\}^d$ is the Hamming cube and $\rho$ is the Hamming distance: the phase transition happens at $\lambda_0=1$,  with the uniform measure $\sigma_H = \frac1{2^d} \sum_{x\in X} \delta_x$  as the unique  maximizer for $0<\lambda < 1$,   measures of the form $\mu = \frac12 (\delta_x + \delta_{-x})$  maximizing the energy for $\lambda >1$, while for $\lambda =1$ maximizers   are all measures with center of mass at $0$.
\item  $X=\mathbb{RP}^d$ and $\rho$ is  the chordal distance: this situation  mirrors the case of the Euclidean distance on the sphere.  There is one transition at $\lambda_0 = 2$, with the unique optimizer $\sigma$ for $-d <\lambda <2$, all isotropic measures being maximizers for $\lambda =2$, and only discrete minimizers $\mu_{ONB}$ for $\lambda >2$, see \cite[Theorem 4.1]{CheHS} and \cite[Theorem 1.1]{AndDGMS} as well as the discussion in Section  \ref{s.chordal} (this result also holds for other projective spaces).
\end{itemize}


\subsection{Atypical behavior of optimizers in the case of the geodesic distance on projective spaces} 
Since the energy  $I_\lambda$ defined in \eqref{e.def} is essentially  the distance integral for $X = \mathbb{RP}^d$ equipped with the geodesic distance,
it might be tempting to suppose that the behavior of its optimizers follows the same pattern. Indeed, if the Fejes T\'oth conjecture is true, there should be a phase transition at $\lambda_0 =1$  (there is a large variety of measures that achieve the same value in \eqref{e.ift} as $\mu_{ONB}$, see e.g. Proposition \ref{p.poleequator} and Lemma \ref{lem:Behavior for geodesic distance for ONB}), and for all $\lambda >1$, it is easy to show that the measure $\mu_{ONB}$ would be the unique maximizer up to rotations and symmetries (see Lemma \ref{lem:Linear programming for ONB}). Notice that an orthonormal basis plays the same role on $\mathbb{RP}^d$ as a pair of opposite poles on $\mathbb S^d$ in the above examples,  in the sense that both are maximal sets of points whose pairwise distances are equal to the diameter of the space.

However, it appears that when it comes to $\lambda < 1$, the behavior of the optimizers of $I_\lambda$ is quite different. {In Section \ref{s.<1},} we shall present numerical evidence which suggests that at least on $\mathbb S^2$, the uniform measure $\sigma$ is not a maximizer of the energy and actual maximizers are measures supported on various smaller sets, and one may perhaps expect a cascade of phase transitions.

Nevertheless, we shall   show that the behavior described above  indeed persists for low   values of $\lambda$, namely, when $- d < \lambda \le  - (d-2)$. When $d\ge 2$, these values are negative, hence, we are interested in {\emph{minimizing}} the energy integral. Notice that the cutoff $\lambda > - d $ is natural as otherwise the energy is  infinite for {every measure} $\mu$. One of the main results of this paper states the following. 
%

\begin{theorem}\label{t.main}
Let $d\ge 2$. For $\lambda  \in (-d , - d+2]$,  the uniform surface measure $\sigma$ on the sphere $\mathbb S^d$  uniquely minimizes the energy integral $I_\lambda (\mu)$ defined in \eqref{e.def}  among all Borel probability measures. 
\end{theorem}


Uniqueness in Theorem \ref{t.main} (as well as in Theorem \ref{t.d1} below)  is understood in the projective sense, up to central symmetries  or up to shifting mass to antipodal points,  or, more precisely, measures $\mu$ and $\nu$ satisfying $\mu  (E \cup -E) =  \nu (E \cup - E)$ for all Borel sets $E\subset \mathbb S^d$ are not distinguished.

The proof of Theorem \ref{t.main}   is presented in Sections \ref{sec:Jacobi Expansion of Geod. Riesz}-\ref{s.singularkernels} in the framework of compact connected two-point homogeneous spaces. In fact, Theorem \ref{t.main} is a consequence  of the much more general Theorem \ref{t.geodriesz} on the geodesic distance Riesz energy on projective spaces, whose proof consists of two major components: positive definiteness/positivity of Jacobi coefficients (Section \ref{sec:Jacobi Expansion of Geod. Riesz}) and approximations arguments   based on Ces\`aro averages of Jacobi series and the spherical analogues of $A_1$ inequalities, which  are designed to take care of the singularities of the kernel (Section \ref{s.singularkernels}). 

\begin{remark}\label{rem.t.main}
We would like to remark that the range $\lambda \in (-d,-d+2]$ in Theorem \ref{t.main} is most likely not sharp. One indication of this is the fact that at the endpoint $\lambda = -d+2$ the uniform measure $\sigma$ is a unique minimizers, and minimizers at phase transitions are usually not unique. 
We conjecture that for all $d\ge 2$ there exists some phase transition value $\lambda_d$ such that the uniform measure  $\sigma$  uniquely optimizes $I_\lambda$ on $\mathbb S^d$ in the range  $-d< \lambda < \lambda_d$, but it does not optimize $I_\lambda$ for $\lambda > \lambda _d$.    In fact, numerical experiments performed by Peter Grabner suggest that $\lambda_2 \ge 0.59$ and $\lambda_3 \ge 0.125$, while it appears that $\lambda_d <0$ for $d\ge 4$. 
\end{remark}

In addition, it is known \cite[Lemma 1.1 and Corollary 3.8]{LimM22} that there exists a value $\lambda^*_d \in [1,2)$ such that $\mu_{ONB}$ uniquely maximizes $I_\lambda$ for $\lambda > \lambda^*$ (in \S\ref{sec.discrete} we partially extend this result to other projective spaces), and the integral Fejes T\'oth conjecture, Conjecture \ref{c.IFT}, essentially states that $\lambda_d^* =1 $. 

Moreover, one can check that $\sigma$ is not a maximizer of $I_1$ for $d\ge 2$, and thus by continuity it doesn't optimize $I_\lambda$ for $\lambda$ close to $1$, see Lemma \ref{l.lambda<1}. Therefore, $\lambda_d <1$. Moreover, $\mu_{ONB}$ is also not an optimizer of $I_\lambda$ for $\lambda <1$, see Proposition \ref{p.notONB}. 

Hence in the range $\lambda_d < \lambda < 1$ neither $\sigma$, nor $\nu_{ONB}$ optimize the energy $I_\lambda$, and the exact nature of optimizers is elusive even on the level of conjectures. In \S\ref{s.<1} we present some observations and numerical experiments which suggest a cascade of  phase transitions with optimal measures of different dimensions.  This drastically differs from the behavior of other similar energies described in \S\ref{s.typical}. For a  visual   reference see Figure \ref{fig:onRP} in \S\ref{s.summary} (notice the parameter used there is $s=-\lambda$).

\subsection{The case $d=1$} The one-dimensional case is special: it is the only case in which optimizers of $I_\lambda$ are completely understood for the whole range $\lambda >-1$. 

\begin{theorem}\label{t.d1}
Let $d=1$, i.e. consider the energy $I_\lambda (\mu)$ on the circle $\mathbb S^1$. 
\begin{enumerate}[(i)]
\item\label{i1} For  $-1 < \lambda \le 0$, the uniform  measure $\sigma$ on  $\mathbb S^1$ is the unique minimizer of $I_\lambda (\mu)$.
\item\label{ii1} For  $0<\lambda <1$, the measure $\sigma$  is the unique maximizer.
\item\label{iii1}  For $\lambda = 1$, maximizers are exactly the orthogonally symmetric measures (see \S\ref{s.d1} for the definition) which include both $\sigma$ and $\mu_{ONB}$. 
\item\label{iv1} For $\lambda >1$,  unique maximizers are measures of the form  $\mu_{ONB}$. 
\end{enumerate}
\end{theorem} 

Uniqueness, just as in Theorem \ref{t.main}, is understood up to symmetries. The presence of $\mu_{ONB}$ in  \eqref{iii1} follows from the fact
that the Fejes T\'oth conjecture is solved on the circle \cite{FodVZ,BilM}, and it easily implies 
part \eqref{iv1}, see Lemma \ref{lem:Linear programming for ONB}. A simple computation showing that $I_{1} (\mu_{ONB}) = I_1 (\sigma) $ on $\mathbb S^1$ justifies the  presence of $\sigma$ in  \eqref{iii1}.

While parts \eqref{i1} and \eqref{ii1} can be addressed in the same way as Theorem \ref{t.main} and they are indeed covered by Theorem \ref{t.geodriesz}, all four parts of Theorem \ref{t.d1} (including the full characterization in \eqref{iii1}) can be viewed as a consequence of the characterization of optimizers of the geodesic Riesz energy on the sphere $\mathbb S^d$ \cite{BilD}  together with the  well-known fact that the real projective line  $\mathbb{RP}^1$ is isometric to the circle $\mathbb S^1$. This approach (which has been overlooked in prior work on the Fejes T\'oth conjecture) is taken in \S\ref{s.d1}, see Theorem \ref{t.d1expand}. 


\subsection{Moving to the real projective space $\mathbb{RP}^{d}$.}\label{s.proj}





In fact, as mentioned earlier, 
since the kernel $\theta(x,y) = \arccos | x\cdot y |$ remains unchanged if $x \in \mathbb S^d$ is changed to $-x$, the natural setting for our problem is the real projective space $\mathbb{RP}^{d}$. 

The real projective space can be defined as $\mathbb{RP}^{d} = \big( \mathbb R^{d+1} \setminus \{0\} \big) / \big( \mathbb R \setminus \{0\} \big) $,  i.e. one identifies elements $ x \in \mathbb R^{d+1} \setminus \{0\} $ with $c x$ for $c\neq 0$. This can be thought of as the set of lines in $\mathbb R^{d+1}$ passing through the origin. Alternatively,  $\mathbb{RP}^{d}$ can be associated with the sphere $\mathbb S^d$ in which the elements $x$ and $-x$ are identified. We shall often use this latter convention, sometimes slightly abusing the notation by not distinguishing between an element $x \in \mathbb{RP}^{d}$ and one of its representers in $\mathbb S^d$. 

The geodesic distance $\vartheta$ on $\mathbb{RP}^{d}$, normalized so that the diameter of the space is equal to $\pi$, is defined by setting 
\begin{equation}\label{e.projgeod}
\cos \vartheta (x,y)  = 2 | x\cdot y |^2 -1 ,
\end{equation}
where in the right-hand side we identify elements of $\mathbb{RP}^{d}$ with their representers in $\mathbb S^d$.  

Letting $\eta(x,y) = \arccos (x\cdot y) \in [0,\pi]$ be the geodesic distance  between  $x$ and $y$ on $\mathbb S^d$, we see  that
\begin{equation}\label{e.thetabeta}
\cos \vartheta (x,y)  = 2 | x\cdot y |^2 -1  = \cos  \big( 2 \eta (x,y) \big) ,
\end{equation}
and therefore
\begin{align}\label{e.rhotheta}
\vartheta (x,y ) & = \arccos  \big( \cos  \big( 2 \eta (x,y) \big) \big)  = \min \{ 2 \eta (x,y), 2\pi - 2\eta (x,y) \} \\
\nonumber & = 2 \min \{ \arccos  x\cdot y , \,  \pi - \arccos  x\cdot y \} = 2 \theta (x,y),
\end{align}
as defined in \eqref{e.theta}. In other words, up to a factor of $2$, the kernel $\theta (x,y)$ arising in the Fejes T\'oth conjecture is the geodesic distance $\vartheta (x,y)$ on the real projective space $\mathbb{RP}^{d}$.

We shall adopt this point of view for the rest of the paper and will study the Riesz energy  with respect to the geodesic distance on the real (as well as other) projective spaces.

To be consistent with the potential theory literature, we define the  geodesic Riesz kernel as 
\begin{equation}\label{e.fs1}
 F_s  (\cos (\vartheta(x,y)) ) = \begin{cases}
    \frac{1}{s}\vartheta^{-s} (x,y) & s \neq 0\\
    - \log( \vartheta(x,y) ) & s = 0
    \end{cases}
\end{equation}
Setting $t = \cos \vartheta(x,y)$, since $\vartheta (x,y) \in [0,\pi]$, we obtain the kernel
\begin{equation}\label{e.fs2}
F_s ( t ) = 
\begin{cases}
   \frac1{s}  \big(  \arccos t \big)^{-s} & s \neq 0,\\
    - \log \big( \arccos t \big) & s = 0.
    \end{cases}
\end{equation}
Notice that $s=-\lambda$ compared to the notation of Theorems \ref{t.main} and \ref{t.d1}. 
Observe two advantages of this setup. First, unlike the kernel $\arccos |t|$ associated to $\theta (x,y) $ on $\mathbb S^d$, the kernel $F_s (t)$ arising from the geodesic distance $\vartheta (x,y)$ on $\mathbb{RP}^{d}$ is smooth inside the interval $(-1,1)$, which simplifies some of the computations. Second,
the presence of the factor $\frac{1}{s}$ implies that we shall be interested in minimizing the energy independently of the sign of $s$. 
From now on we shall use this notation (although we shall briefly return to $\lambda = -s$ in \S\ref{s.<1}).

\subsection{Geodesic distance Riesz energy on projective spaces. } Motivated by the above discussion we investigate the minimizers of the geodesic distance Riesz energy on compact connected two-point homogeneous spaces. These spaces are completely characterized \cite{Wang}: they include the sphere $\mathbb S^d$ and the projective spaces $\mathbb{RP}^d$, $\mathbb{CP}^d$, $\mathbb{HP}^d$ ($d\ge 1 $), as well as $ \mathbb{OP}^d$ ($d=1$ or $2$) over the fields of real numbers, complex numbers, quaternions, and octonions, respectively (we shall use the notation $\mathbb{FP}^d$ when we want to collectively refer to these projective spaces). Rather than giving precise definitions of these spaces, we refer the reader to one of the standard sources, e.g. \cite{Helg} (we would also like to point out a recent reference \cite{AndDGMS} where these spaces were studied in the context of energy minimization).  From now on $\Omega$ will always denote one of the spaces described above.\\



In fact, since the geodesic Riesz energy on the sphere has been studied in detail in \cite{BilD},  we  mostly concentrate on the projective spaces (although, the results on the sphere  fall into the same general framework, i.e. our arguments automatically reproduce them). 
While the case of  one-dimensional  projective spaces  is also covered by the results of this section, a more thorough understanding of this case follows  from the known results for the sphere using the isometry of $\mathbb{FP}^1$ and $\mathbb S^{\dim_{\mathbb R} \mathbb F}$, see Section \ref{s.d1}.  Finally, we would like to add that, since $\mathrm{SO}(3)$ is isometric to $\mathbb{RP}^3$, see e.g. \cite[Section 2]{Graf}, the results of this section also apply to this space. In the case $s=-1$ these energies on all projective spaces have been studied in \cite{Gan}, where it was shown that $\sigma$ does not minimize them, see Lemma \ref{l.lambda<1}.  \\

We now state the main  theorem which provides sufficient conditions for  the geodesic Riesz $I_{F_s}$ energies, as defined in \eqref{eq:Energy def}, to be  minimized by the uniform measure $\sigma$.  Let $D = d \cdot {\dim_{\mathbb R} \mathbb F}$ denote the dimension of $\Omega$ as a real manifold.

\begin{theorem}\label{t.geodriesz}
Let $\Omega$ be a compact connected 
 two-point homogeneous space, and let the parameter $s \in \mathbb R$ satisfy the following, depending on $\Omega$:
 \begin{itemize}
\item $\Omega = \mathbb{S}^d$ and $- 1 < s < d$,

\item $\Omega = \mathbb{RP}^d$, $d \geq 2$ and $d -2  \leq s < d$,  or $\Omega = \mathbb{RP}^1$ and $-1 < s < 1$,

\item $\Omega = \mathbb{CP}^d$, $d \geq 2$, and $ 2d-3 \leq s < D =2d $, or $\Omega = \mathbb{CP}^1$ and $-1 < s < 2$,

\item $\Omega = \mathbb{HP}^d$, $d \geq 2$, and $4d - 5 \leq s < D=4d$, or $\Omega = \mathbb{HP}^1$ and $-1 < s < 4$ 

\item $\Omega = \mathbb{OP}^2$ and $7\leq s < D= 16$, or $\Omega = \mathbb{OP}^1$ and $-1 < s  < 8$

\end{itemize}

\noindent Then the uniform measure $\sigma$  {uniquely} 
 minimizes the geodesic distance Riesz energy $I_{F_s} (\mu)$ among  Borel probability measures on $\Omega$. 

In all of the above examples, if $d=1$ (or $\Omega= \mathbb S^d$) and $s=-1$, the uniform measure $\sigma$ minimizes the geodesic Riesz energy, but it is not a unique minimizer. 
\end{theorem}

Observe that this theorem with $\Omega = \mathbb{RP}^d$ implies Theorem \ref{t.main} and parts \eqref{i1}-\eqref{ii1}, as well as partially  \eqref{iii1}, of Theorem \ref{t.d1}.  The ranges of $s$ for $\Omega = \mathbb{FP}^d$ with $d\ge 2$ is most likely not sharp (see Remark \ref{rem.t.main}). Finding the exact transition value $s_d$ for $ \mathbb{FP}^d$ such that the uniform measure minimizes the energy $I_{F_s}$ for $s_d <s < d$, but not for $s<s_d$ seems to be a very interesting, but difficult open problem.

\subsection{Summary of the differences of Riesz energies on the sphere and on projective spaces}\label{s.summary}

The case $\Omega = \mathbb S^d$ of Theorem \ref{t.geodriesz} recovers the results of \cite{BilD}. Notice that this is the only $\Omega$ in Theorem \ref{t.geodriesz} with $d\ge 2$  in which the minimizing measures are known  for the whole admissible range of exponents $s<d$. Indeed, in addition to the case $-1<s<d$ covered above,  for $s=-1$ the minimizers are all centrally symmetric measures, and for $s<-1$ one obtain discrete measures concentrated equally at two opposite poles $\frac12(\delta_x + \delta_{-x})$, see \cite[Theorem 1.1]{BilD}. As discussed in \S\ref{s.typical}, the minimizers of the Euclidean distance Riesz energy on $\mathbb S^d$ demonstrate the same pattern of behavior with a single phase transition at $s= -2 $    \cite[Theorems 5 and 7]{Bj}, \cite[Theorems 4.6.4 and 4.6.5]{BorHS}. This is summarized in Figure \ref{fig:onRP}. 
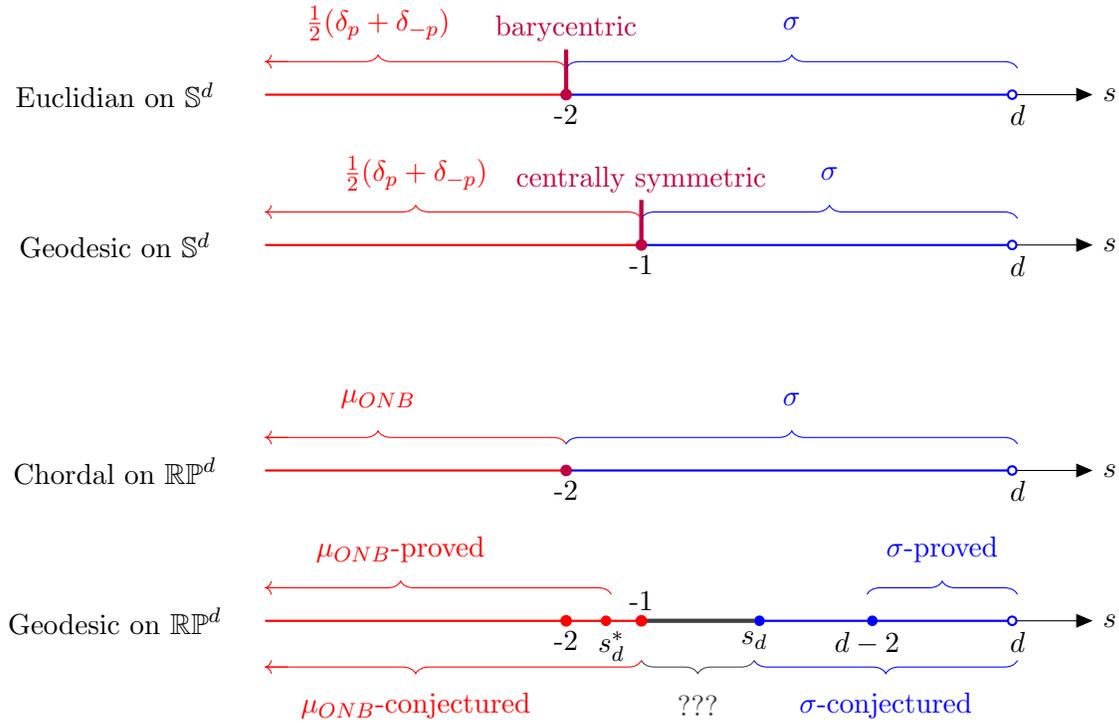
\begin{figure}[h!]
\centering
  \begin{tikzpicture}[
line/.style = {draw,thick}]
\draw[>=triangle 45, ->] (-6,0) -- (5,0) node[right]{$s$};
\draw[blue,line, -{Circle[length=4pt,fill=white]}]  (-2,0) -- (4,0) node[black,below]{$d$};
\draw[red,line, -]  (-6,0) -- (-2,0) node[black,below]{-2} node[yshift=+25pt,purple]{barycentric};
\filldraw[purple] (-2,0) circle (2pt);
\draw[purple,line,ultra thick] (-2,0)--(-2,0.6);
\draw[blue,decorate, decoration = {brace, amplitude = 5pt}, xshift = 0pt, yshift = +10pt] (-2,0)--(4,0) node [blue, midway, xshift = 0cm, yshift = +0.6cm] {$\sigma$};
\begin{scope}
        \clip(5, 4) rectangle (-6, -2);
        \draw [red,decorate,decoration={brace,amplitude=5pt},xshift= 0pt, yshift=+10pt] (-7,0)--(-2,0) node [midway, xshift = 0cm, yshift = +0.6cm] {$\frac{1}{2}(\delta_p+\delta_{-p})$};
\end{scope}
\draw[<-,red] ([yshift=12.5pt]-6, 0) -- ([yshift=12.5pt]-5.7,0);
\node[] at (-8,0) {Euclidian on $\mathbb{S}^d$};

\draw[>=triangle 45, ->] (-6,-2) -- (5,-2) node[right]{$s$};
\draw[blue,line, -{Circle[length=4pt,fill=white]}]  (-1,-2) -- (4,-2) node[black,below]{$d$};
\draw[red,line, -]  (-6,-2) -- (-1,-2) node[black,below]{-1} node[yshift=+25pt,purple]{centrally symmetric};
\filldraw[purple] (-1,-2) circle (2pt);
\draw[purple,line,ultra thick] (-1,-2)--(-1,-1.4);
\draw[blue,decorate, decoration = {brace, amplitude = 5pt}, xshift = 0pt, yshift = +10pt] (-1,-2)--(4,-2) node [blue, midway, xshift = 0cm, yshift = +0.6cm] {$\sigma$};
\begin{scope}
        \clip(5, -4) rectangle (-6, 1);
        \draw [red,decorate,decoration={brace,amplitude=5pt},xshift= 0pt, yshift=+10pt] (-7,-2)--(-1,-2) node [midway, xshift = 0cm, yshift = +0.6cm] {$\frac{1}{2}(\delta_p+\delta_{-p})$};
\end{scope}
\draw[<-,red] ([yshift=12.5pt]-6, -2) -- ([yshift=12.5pt]-5.7,-2);
\node[] at (-8,-2) {Geodesic on $\mathbb{S}^d$};

\draw[>=triangle 45, ->] (-6,-5) -- (5,-5) node[right]{$s$};
\draw[blue,line,-{Circle[length=4pt,fill=white]}]  (-2,-5) -- (4,-5) node[black,below]{$d$};
\draw[red,line, -]  (-6,-5) -- (-2,-5) node[black,below]{-2} node[yshift=+20pt,purple]{};
\filldraw[purple] (-2,-5) circle (2pt);
\draw[blue,decorate, decoration = {brace, amplitude = 5pt}, xshift = 0pt, yshift = +10pt] (-2,-5)--(4,-5) node [blue, midway, xshift = 0cm, yshift = +0.6cm] {$\sigma$};
\begin{scope}
        \clip(5, -1) rectangle (-6, -6);
        \draw [red,decorate,decoration={brace,amplitude=5pt},xshift= 0pt, yshift=+10pt] (-7,-5)--(-2,-5) node [midway, xshift = 0cm, yshift = +0.6cm] {$\mu_{ONB}$};
\end{scope}
\draw[<-,red] ([yshift=12.5pt]-6, -5) -- ([yshift=12.5pt]-5.7,-5);
\node[] at (-8,-5) {Chordal on $\mathbb{RP}^d$};

\draw[>=triangle 45, ->] (-6,-7) -- (5,-7) node[right]{$s$};
\draw[blue,line, {Circle[length=4pt]}-{Circle[length=4pt,fill=white]}]  (2,-7) node[black,below]{$d-2$} -- (4,-7) node[black,below]{$d$};
\draw[blue,line, {Circle[length=4pt]}-] (0.5,-7) node[black,below]{$s_d$} -- (2,-7);
\draw[red,line, -{Circle[length=4pt]}]  (-6,-7) -- (-1.4,-7) node[black,below]{$s_d^*$} node[yshift=+30pt]{};
\draw[red,line] (-1.4,-7)--(-1,-7) node[black,above] {-1};
\draw[darkgray,line,ultra thick] (-1,-7)--(0.5,-7);
\filldraw[red] (-1,-7) circle (2pt);
\filldraw[red] (-2,-7) circle (2pt) node[black,below]{-2};
\draw[blue,decorate, decoration = {brace, amplitude = 5pt}, xshift = 0pt, yshift = +10pt] (2,-7)--(4,-7) node [blue, midway, xshift = 0cm, yshift = +0.6cm] {$\sigma$-proved};
\draw[blue,decorate, decoration = {brace, amplitude = 5pt,mirror}, xshift = 0pt, yshift = -15pt] (0.5,-7)--(4,-7) node [blue, midway, xshift = 0cm, yshift = -0.6cm] {$\sigma$-conjectured};
\draw[darkgray,decorate, decoration = {brace, amplitude = 5pt,mirror}, xshift = 0pt, yshift = -15pt] (-1,-7)--(0.5,-7) node [darkgray, midway, xshift = 0cm, yshift = -0.6cm] {???};
\begin{scope}
        \clip(6, -10) rectangle (-6, -4);
        \draw [red,decorate,decoration={brace,amplitude=5pt},xshift= 0pt, yshift=+10pt] (-7,-7)--(-1.4,-7) node [midway, xshift = 0cm, yshift = +0.6cm] {$\mu_{ONB}$-proved};
        \draw [red,decorate,decoration={brace,amplitude=5pt,mirror},yshift=-15pt] (-7,-7) -- (-1,-7) node[midway,xshift=0cm,yshift=-0.6cm] {$\mu_{ONB}$-conjectured};
\end{scope}
\draw[<-,red] ([yshift=12.5pt]-6, -7) -- ([yshift=12.5pt]-5.7,-7);
\draw[<-,red] ([yshift=-17.5pt]-6, -7) -- ([yshift=-17.5pt]-5.7,-7);
\node[] at (-8,-7) {Geodesic on $\mathbb{RP}^d$};
\end{tikzpicture}
\label{fig.numberlines}
\caption{Minimizers of the Euclidean and geodesic Riesz energies on the sphere $\mathbb S^d$ and real projective space $\mathbb{RP}^d$.}
\label{fig:onRP}
\end{figure}

As we shall see in \S\ref{s.chordal} (Theorem \ref{thm:Chordal Riesz Minimizers}), the  minimizers of the Riesz energy with respect to chordal distance on $\mathbb{FP}^d$  behave in exactly the same way as those of the Euclidean distance Riesz energy on $\mathbb S^d$, i.e. a single phase transition at $s=-2$. It is therefore surprising that the geodesic Riesz energy on $\mathbb{FP}^d$ behaves quite differently from its counterpart on the sphere. In Theorem \ref{t.FTlarge}  we show that there exists $s^*_d \in [-2,-1]$ such that $\mu_{ONB}$ is a unique minimizer  of $I_{F_s}$ for $s<s^*_d$, but not for $s>s^*_d$. Then, just as explained in Remark \ref{rem.t.main} and the discussion thereafter, on the nonempty interval $-1 < s < s_d$ neither $\sigma$ nor $\mu_{ONB}$ minimizes the geodesic distance Riesz energy on $\mathbb{FP}^d$. The structure of minimizers in this range remains mysterious, with numerical results of \S\ref{s.<1} suggesting  further phase transitions for $s>-1$. We illustrate this effect on $\mathbb{RP}^d$ in Figure \ref{fig:onRP}.



\section{Preliminaries: analysis on compact connected two-point homogeneous spaces, Jacobi polynomials}\label{s.prelim}

We briefly discuss relevant analysis on projective spaces, for more details see  \cite{AndDGMS}.  Let $\Omega$ denote $\mathbb S^d$, $\mathbb{RP}^d$, $\mathbb{CP}^d$, $\mathbb{HP}^d$ ($d\ge 1 $), or  $ \mathbb{OP}^d$ ($d=1$ or $2$). Each of these spaces admits a geodesic metric $\vartheta(x,y)$ normalized so that $\operatorname{diam} (\Omega) =\pi$ and a chordal metric, which on the projective spaces is standardly  given by 
$$ \rho(x,y) = \sqrt{\frac{1-\cos \big(\vartheta(x,y)\big)}{2}} = \sin \frac{\vartheta(x,y)}{2},$$
 and on the sphere $\mathbb S^d$ its role is played by the Euclidean distance $ \rho(x,y) = \| x-y\| = 2 \sin \frac{\vartheta(x,y)}{2}$. 

With each space, we associate parameters
\begin{equation}
    \label{eq:JacobiAlphaBeta}
    \alpha = \begin{cases} \frac{d}{2} -1, \\ d \frac{\dim_{\mathbb{R}} (\mathbb{F})}{2}-1, \end{cases}  \; \; \; \;
    \beta = \begin{cases}
        \alpha, & \text{ for } \Omega = \mathbb{S}^{d};\\
        \frac{\dim_{\mathbb{R}}(\mathbb{F})}{2} -1, &
        \text{ for } \Omega = \mathbb{FP}^{d}.
    \end{cases}
\end{equation}
As before, we use $D=\dim_{\mathbb R} (\Omega)= d\cdot \dim_{\mathbb{R}} (\mathbb{F})  = 2\alpha+2$  to denote the dimension of the space $\Omega$ as a real manifold. We also denote
\begin{equation}
d\nu^{(\alpha, \beta)}(t)  :=  \frac1{c_{\alpha,\beta}} (1-t)^{\alpha} (1+t)^{\beta} dt,
\end{equation}
where
\begin{equation}
c_{\alpha, \beta} = 2^{\alpha +\beta +1}B(\alpha+1, \beta+1)
\end{equation}
with $B$ denoting  the beta function. For each space, we will also define $\sigma = \sigma_{\alpha, \beta}$ to be the uniform probability measure on $\Omega$.

We recall that for any function $f:[-1,1] \rightarrow \mathbb R \cup \{ +\infty\}$ which is integrable with respect to the measure $\nu^{(\alpha, \beta)}$, one has the identity 
\begin{equation}\label{eq:Equivalence with 1 dim setting}
\int_{\Omega}  f( \cos( \vartheta(x,y)) ) d\sigma (x) = \int_{-1}^1  f(t) d\nu^{\alpha, \beta}(t)
\end{equation}
for every $y\in \Omega$.

The Jacobi polynomials $P_n^{(\alpha, \beta)} (t)$ are a sequence of polynomials orthogonal with respect to the weight $\nu^{(\alpha, \beta)}$, for $\alpha, \beta > -1$, normalized so that \begin{equation}
P_n^{(\alpha, \beta)}(1) = \binom{n+\alpha}{n}. 
\end{equation}
In the case when $\Omega = \mathbb S^d$, i.e. $\alpha =\beta = \frac{d}{2}-1$, Jacobi polynomials, up to normalization, coincide with the Gegenbauer (ultraspherical) polynomials $C_n^{\alpha + 1/2}$, which arise naturally in the analysis on the sphere.

\begin{remark} We would like to note that due to the identity $$ C_{2n}^{\alpha + 1/2} (t)  = \frac{(\alpha + 1/2)_n}{(1/2)_n}  P_n^{(\alpha,\beta)} (2t^2 -1),$$
 by a change of variables,  the analysis of kernels on $\Omega = \mathbb{RP}^d$  using Jacobi polynomials could be equivalently carried out as analysis of even kernels on $\mathbb S^d$ using Gegenbauer polynomials. While this was our initial approach for the energy $I_\lambda$ defined in \eqref{e.def}, Jacobi polynomials provide a more general framework for all projective spaces.  
 \end{remark}
 
 Every function $ f \in L^2 \big( [-1,1],    \nu^{(\alpha, \beta)} \big)$, can be represented by its Jacobi series 
\begin{equation}\label{eq:Jacobi Expansion}
f (t) = \sum_{n=0}^\infty \widehat{f}_n   P_n^{(\alpha, \beta)} (t) ,
\end{equation}
where the convergence is in the $L^2$ sense. 
The Jacobi  coefficients $ \widehat{f}_n$ are defined for any function $ f \in L^1 \big( [-1,1],    \nu^{(\alpha, \beta)} \big)$ (although the series  \eqref{eq:Jacobi Expansion} doesn't necessarily converge in this case) and are  computed according to the formula
\begin{equation}\label{eq:Jacobi Coefficients}
\widehat{f}_n = \frac{ m_n^{(\alpha, \beta)}}{ \big( P_n^{(\alpha, \beta)}(1)\big)^2} \int_{-1}^1  f(t) P_n^{(\alpha, \beta)} (t) \, d\nu^{(\alpha, \beta)}(t),
\end{equation}
where
\begin{equation}
m_n^{(\alpha,\beta)} := \frac{2n + \alpha + \beta + 1}{\alpha + \beta + 1} \frac{(\alpha + \beta + 1)_n (\alpha + 1)_{n}}{n! (\beta + 1)_n},
\end{equation}
and  $(k)_n$ is the Pochhammer notation for the rising factorial:
\begin{equation}
(k)_n := \prod_{j=0}^{n-1} (k+j).
\end{equation}

Jacobi polynomials are positive definite on their corresponding spaces. More precisely, for any $n \in \mathbb N_0$ and any finite signed Borel measure $\mu$ on $\Omega$
\begin{equation}\label{eq:Jacobi Pos def}
\int_{\Omega} \int_{\Omega} P_n^{(\alpha, \beta)} (\cos(\vartheta(x,y)) \,  d\mu (x) d\mu (y) \ge 0.
\end{equation}
This property can be taken as the definition of positive definiteness ({\it{conditional}} positive definiteness means that the inequality above holds for all signed measures which have total mass zero, and {\it{strict}} positive definiteness means that equality can only hold if $\mu$ is zero on all Borel sets. For general kernels, these definitions involve only measures for which the energy integral is well defined).

Moreover, Jacobi polynomials  span the positive cone of isometry-invariant positive definite continuous kernels on $\Omega$, in other words, such a kernel is positive definite if and only if all of its Jacobi coefficients are nonnegative, see e.g. \cite[Corollary 3.22 and pp. 176--179]{Gan}, 
which in the case of $\mathbb{S}^d$ goes back to Schoenberg \cite{Sch}.

In addition, it is known that, if $f\in C[-1,1]$, i.e. $f$ is continuous,  and  all the Jacobi coefficients of $f$ are nonnegative, then the Jacobi expansion \eqref{eq:Jacobi Expansion}  of $f$  converges to $f$ absolutely and uniformly due to Mercer's Theorem, see \cite[Lemma 2.16]{AndDGMS} for details. 

These properties naturally connect  Jacobi expansions to energy minimization. For a  function $f:[-1,1] \rightarrow \mathbb R\cup \{+\infty\}$, which is bounded below, we define the energy of a Borel  probability measure $\mu$ as 
\begin{equation}\label{eq:Energy def}
I_f (\mu) = \int_{\Omega} \int_{\Omega}  f (\cos(\vartheta(x,y)))  d\mu (x) d\mu (y).
\end{equation}

The properties discussed  above  allow one to characterize those continuous kernels $f$ for which the energy $I_f(\mu)$ is minimized by the uniform measure $\sigma$:

\begin{theorem}\label{thm:uniform measure min, continuous}
Assume that $f\in C[-1,1]$.  The following conditions are equivalent:
\begin{itemize}
\item The energy $I_f (\mu)$ is minimized by the uniform measure $\sigma$, i.e. for every Borel probability measure on $\Omega$ 
$$I_f(\mu) \ge I_f (\sigma).$$
\item The kernel $f(\cos(\vartheta(x,y)))$ is positive definite up to an additive constant on $\Omega$.
\item The kernel $f(\cos(\vartheta(x,y)))$ is conditionally positive definite.
\item The Jacobi coefficients of $f$ are nonnegative, i.e. $\widehat{f}_n \ge 0$, for all $n\ge 1$.
\end{itemize}
The following conditions are also equivalent:
\begin{itemize}
\item The energy $I_f (\mu)$ is uniquely minimized by the uniform measure $\sigma$.
\item The kernel $f(\cos(\vartheta(x,y)))$ is strictly positive definite up to an additive constant on $\Omega$.
\item The kernel $f(\cos(\vartheta(x,y)))$ is conditionally strictly positive definite.
\item The Jacobi coefficients of $f$ are positive, i.e. $\widehat{f}_n > 0$, for all $n\ge 1$.
\end{itemize}
\end{theorem}

This result can be found in various forms in, for instance, \cite[Lemma 2, Theorem 2, Section III]{Boc}, \cite[Corollary 2.15, Theorem 2.17]{AndDGMS}, \cite[Theorems 4.1 \& 4.3]{BilMV}, \cite[Proposition 2.3]{BilGMPV}, \cite{Sch}. \\

As suggested by Theorem \ref{thm:uniform measure min, continuous}, in order to prove our main result, Theorem \ref{t.geodriesz}, for $d\ge 2$, we shall demonstrate  positivity of the Jacobi coefficients of the Riesz kernels $F_s$ for the corresponding ranges of $s\in \mathbb R$.  This will be done in Section  \ref{sec:Jacobi Expansion of Geod. Riesz}. \\

This, however, will not suffice, since Theorem \ref{thm:uniform measure min, continuous} cannot be applied directly. 
Indeed, the conditions of Theorem \ref{thm:uniform measure min, continuous} require continuity of the function $f$.
Generally, when $f$ is  singular (which is the case for the kernels $F_s$ with $s\geq 0$), positivity of Jacobi coefficients alone is not enough to deduce that the energy is minimized by the uniform measure $\sigma$ -- some approximation arguments are required in addition. Indeed, in many similar situations (see e.g. \cite{AndDGMS,BilD,BorHS,DamG}) one removes the singularity in an \emph{ad hoc} fashion by  adding an $\varepsilon > 0$ in the denominator or (almost equivalently) considering a ``shifted'' non-singular version of the kernel $f(t-\varepsilon)$, and then taking the limit as $\varepsilon \rightarrow 0$. Unfortunately, none of these simple approximation methods appear to work in our setting, since  our argument establishing positivity of Jacobi coefficients collapses under  such perturbations. \\

In Section \ref{s.singularkernels} we shall present more elaborate approximation arguments based on two harmonic analysis methods:  Ces\`aro averages of Jacobi expansions and  $A_1$ inequalities for maximal functions. These approximations will allow us to conclude that strict  positivity of Jacobi coefficients for the geodesic Riesz kernels indeed implies that the uniform measure  uniquely minimizes the corresponding Riesz energies, thus proving the main result, Theorem \ref{t.geodriesz}.

Moreover, as a corollary, in Theorem \ref{thm:general} we establish a rather general condition on the singular kernel, under which positivity of Jacobi coefficients implies that the uniform measure  $\sigma$ minimizes the energy. To the best of our knowledge, this is the first general result of this kind as most of the prior work mentioned above concentrated on concrete kernels.

%



\section{Positivity of the  Jacobi coefficients}\label{sec:Jacobi Expansion of Geod. Riesz}
As an important first step towards the proof of Theorem \ref{t.geodriesz}, we begin by establishing the positivity of  the Jacobi coefficients of the projective geodesic Riesz kernel defined in \eqref{e.fs1}--\eqref{e.fs2}.  In the rest of this section we assume that $\alpha$ and $\beta$ take geometrically relevant values as defined in \eqref{eq:JacobiAlphaBeta}, in particular, $\alpha, \beta \ge -\frac12$.

According to \eqref{eq:Jacobi Coefficients}, it suffices to prove the following.

\begin{theorem}\label{thm:Pos Jacobi Coefficients Geod Riesz}
 For all  $n \geq 1$,
\begin{equation}\label{eq:jacCoeffs}
\int_{-1}^1 F_s (t) P_n^{(\alpha, \beta)}(t) (1-t)^{\alpha} (1+t)^{\beta} dt > 0
\end{equation}
if  $2(\alpha - \beta) - 1 \leq  s < D$ and $s>-1$. 
\end{theorem}

\begin{remark}For each of the compact connected two-point homogeneous spaces, this result translates exactly to the bounds stated in Theorem \ref{t.geodriesz}:
\begin{itemize}
\item $\Omega = \mathbb{S}^d$ and $- 1 < s < d$,

\item $\Omega = \mathbb{RP}^d$, $d \geq 2$ and $d -2  \leq s < d$,  or $\Omega = \mathbb{RP}^1$ and $-1 < s < 1$,

\item $\Omega = \mathbb{CP}^d$, $d \geq 2$, and $ 2d-3 \leq s < D =2d $, or $\Omega = \mathbb{CP}^1$ and $-1 < s < 2$,

\item $\Omega = \mathbb{HP}^d$, $d \geq 2$, and $4d - 5 \leq s < D=4d$, or $\Omega = \mathbb{HP}^1$ and $-1 < s < 4$ 

\item $\Omega = \mathbb{OP}^2$ and $7\leq s < D= 16$, or $\Omega = \mathbb{OP}^1$ and $-1 < s  < 8$

\end{itemize}
We note that  restriction $d -2  \leq s < d$ arising in in the real projective spaces is responsible for the range of validity of Theorem \ref{t.main}.
\end{remark}

The nature of the condition $s\geq  2(\alpha - \beta) - 1 $ in Theorem \ref{thm:Pos Jacobi Coefficients Geod Riesz}
  will become clear in the proof of Proposition \ref{prop:Positivitiy of odd terms}, while the condition $s>-1$ arises in Proposition \ref{prop:FDerivatives} and Lemma \ref{lem:Fa IBP}.  Observe that the latter condition forces the interval of admissible values of $s$ to be open in the case of $\Omega = \mathbb S^d$ or $\Omega = \mathbb{FP}^1$. 
  
 \begin{remark}\label{rem.d=1 nonnegative}
 While the value $s=-1$ is excluded in Theorem \ref{thm:Pos Jacobi Coefficients Geod Riesz} in the cases $\Omega = \mathbb S^d$ or $\Omega = \mathbb{FP}^1$, since the functions $F_s$ are continuous for $s<0$ and converge uniformly to $F_{-1}$ as $s\rightarrow -1$ from the right, it is easy to see from \eqref{eq:jacCoeffs} that the Jacobi coefficients are non-negative in this case. 
 \end{remark}

  \begin{remark} The case $\Omega = \mathbb S^d$ has been essentially handled in \cite{BilD}.  Observe that this is is the only example in the class of compact connected two-point homogeneous spaces with $d>1$, where non-negativity of the Jacobi coefficients of the geodesic Riesz kernel extends all the way to $s=-1$, thus leading to a single phase transition at this value. 

\end{remark}

Before we prove Theorem \ref{thm:Pos Jacobi Coefficients Geod Riesz}, we introduce a series of auxiliary  statements.
For $s> - 1$, and $ b, c \ge 0 $,  let us define
\begin{equation}
F_{s,b,c}(t) \coloneqq \begin{cases}
\frac{t^c(\arccos t )^{-s}}{(1-t^2)^b} & s>0,\\
-\frac{t^c \log(\arccos t)}{(1-t^2)^b} & s = 0, \\
-\frac{t^c(\arccos t )^{-s}}{(1-t^2)^b} &-1<s<0.
\end{cases}
\end{equation}
Then $F_{s,0,0}  (t) = s F_s (t) $ for $ s> 0$, $F_{0,0,0}  (t) = F_0 (t)$, and $F_{s,0,0} = -s F_s(t)$ for $-1<s<0$.  We first show that the   derivatives of $F_s$ can be represented as positive linear combinations of the functions defined above. 

\begin{proposition}\label{prop:FDerivatives}
For all $k \geq 0$ and $s>-1$, we have 
\begin{equation}\label{eq:FkthDerivative}
\frac{d^k}{dt^k}F_{s}(t) = \sum_{i=1}^{3^k} C_i F_{s_i,b_i,c_i}(t),
\end{equation}
where for all $i=1,\dots, 3^k$, $C_i \geq 0$, $c_i \in \mathbb{N}_0$, $0 \leq b_i \leq k- \frac{1}{2}$, and $b_i + \frac{s_i}{2} \leq k+ \frac{s}{2}$. Additionally, when $k \geq 1$, $s_i \geq s + 1$.
\end{proposition}
Note that the proof also implies that for each $k \in \mathbb{N}$ there are some $i, j \in \{ 1, ..., 3^k\}$ such that $s_i = s+1$ and $b_j + \frac{s_j}{2} = k+ \frac{s}{2}$, so these bounds cannot be improved.
\begin{proof}
Since $\frac{d}{dt}F_{s}(t) =   F_{s+1,1/2,0}$  for $s \neq 0$   
and $\frac{d}{dt} F_0(t) = F_{1,1/2,0}$, the statement  is true for $k=0$ and $k=1$. We then use induction with the base case $k=1$, where, since $s > -1$,  the indices $s_i$ are all going to be positive. The facts that  $C_i \geq 0$, $s_i \geq s + 1$, $0 \leq b_i \leq k- \frac{1}{2}$ and $c_i \in \mathbb{N}_0$ follow easily from
\begin{equation}\label{eq:FDerivative}
\frac{d}{dt}F_{\gamma,b,c}(t) = \gamma F_{\gamma+1,b+1/2,c} + 2b F_{\gamma,b+1,c+1} + c F_{\gamma,b,c-1} \qquad \gamma, b,c \geq 0.
\end{equation} 
Finally,  we notice that the quantity $b + \frac{\gamma}{2}$ increases by at most $1$ in each term of the derivative, thus we obtain $b_i + \frac{s_i}{2} \leq k+ \frac{s}{2}$.
\end{proof}

Using the Rodrigues formula for Jacobi polynomials
\begin{equation}\label{eq:Rodrigues}
P_n^{(\alpha, \beta)}(t) (1-t)^{\alpha}(1+t)^{\beta} =  \frac{(-1)^n}{2^n n!}\frac{d^n}{dt^n} \left((1-t)^{n+\alpha}(1+t)^{n+\beta}\right),
\end{equation}
we can rewrite the integral  \eqref{eq:jacCoeffs} as 
\begin{equation}\label{eq:RodriguesIntegral}
\int_{-1}^1 F_s(t) \frac{(-1)^n}{2^n n!}\frac{d^n}{dt^n} \left((1-t)^{n+\alpha}(1+t)^{n+ \beta}\right) dt.
\end{equation}
The following lemma sets up the integration by parts for  this integral. We note that this lemma does not hold in the case $s \leq -1$.

\begin{lemma}\label{lem:Fa IBP}
Let $-1< s< D$ and $n \in \mathbb{N}$. Then for $0 \le k \le n -1 $, we have 
\begin{equation}\label{eq:Fa IBP endpoints zeros}
\left(\frac{d^k}{dt^k} F_s(t)\right)\left(\frac{d^{n-k-1}}{dt^{n-k-1}}\left((1-t)^{n+\alpha}(1+t)^{n+\beta}\right)\right) \Big|_{-1}^1 = 0,
\end{equation}
and for $0\le k \le n$
\begin{equation}\label{eq:Fa IBP Integrable claim}
\left(\frac{d^k}{dt^k} F_s(t)\right)\left(\frac{d^{n-k}}{dt^{n-k}}\left((1-t)^{n+\alpha}(1+t)^{n+\beta}\right)\right)
\end{equation}
is integrable. 
\end{lemma}

\begin{proof}

We note that for any $n \in \mathbb{N}_0$ and $\ell \in \{ 0, ..., n\}$, there exist nonnegative constants $D_0, ..., D_{n- \ell}$ such that 
\begin{equation}\label{eq:Fa IBP Derivative of weights}
\left| \frac{d^{n-\ell}}{dt^{n- \ell}} \left((1-t)^{n+\alpha}(1+t)^{n+\beta}\right) \right| \leq \sum_{j=0}^{n-\ell} D_j (1-t)^{n+ \alpha - j} (1+t)^{\beta + \ell + j}.
\end{equation}

In the case that $-1 < s < 0$ (when the kernel $F_s$ is bounded)  and $k = 0$, we have that
\begin{align*}
\left| F_s(t) \frac{d^{n-1}}{dt^{n-1}}\left((1-t)^{n+\alpha}(1+t)^{n+\beta} \right) \right| & \leq \frac{\pi^{-s}}{|s|}  \sum_{j=0}^{n-1} D_j (1-t)^{n+ \alpha - j} (1+t)^{\beta + 1 + j}
\end{align*}
and
\begin{align*}
\left| F_s(t) \frac{d^{n}}{dt^{n}}\left((1-t)^{n+\alpha}(1+t)^{n+\beta} \right) \right| & \leq \frac{\pi^{s}}{|s|}   \sum_{j=0}^{n} D_j (1-t)^{n+ \alpha - j} (1+t)^{\beta + 1 + j}
\end{align*}
and the claim follows from the fact that $\alpha, \beta > -1$.

We next handle the case $s = 0$, $k = 0$. On $[-1, \cos(1)]$, $|F_0(t)| \leq \log \pi$, and since
\begin{equation}\label{eq:Fa IBP bound on arccos reciprocal}
\arccos t \ge \sqrt{2-2t}
\end{equation}
for $t\in [-1, 1]$ (the easiest way to see it is to observe that the geodesic distance on $\mathbb S^1$ is greater than Euclidean), we have that
\begin{equation*}
\left| F_0(t) \right| \leq \log \frac{1}{\sqrt{1-t}}  = - \frac{1}{2} \log(1-t) 
\end{equation*}
on $[\cos(1), 1)$. Combining these bounds with \eqref{eq:Fa IBP Derivative of weights} , we have
\begin{align*}
\left| F_0(t) \frac{d^{n-1}}{dt^{n-1}}\left((1-t)^{n+\alpha}(1+t)^{n+\beta} \right) \right| & = | F_0(t)|  \sum_{j=0}^{n-1} D_j (1-t)^{n+ \alpha - j} (1+t)^{\beta + 1 + j}\\
& \leq \begin{cases}
C_1 (1-t)^{n + \alpha} (1+t)^{\beta +1} & t \in (-1, \cos(1)] \\
C_2 \big| \log(1-t) \big| (1-t)^{\alpha + 1} (1+t)^{n+\beta} & t \in [\cos(1), 1) 
\end{cases}
\end{align*}
and
\begin{align*}
\left| F_0(t) \frac{d^{n}}{dt^{n}}\left((1-t)^{n+\alpha}(1+t)^{n+\beta} \right) \right| & = | F_0(t)|  \sum_{j=0}^{n} D_j (1-t)^{n+ \alpha - j} (1+t)^{\beta + 1 + j}\\
& \leq \begin{cases}
C_1 (1-t)^{n + \alpha} (1+t)^{\beta } & t \in (-1, \cos(1)] \\
C_2 \big| \log(1-t) \big|(1-t)^{\alpha} (1+t)^{n+\beta} & t \in [\cos(1), 1) 
\end{cases}.
\end{align*}
Our claim now follows from the fact that $\alpha, \beta \geq -\frac{1}{2}$.

For the remaining cases ($ s>  0$ or $k \geq 1$) we find, using Proposition \ref{prop:FDerivatives} together with \eqref{eq:Fa IBP bound on arccos reciprocal}, that there exist nonnegative constants $A_1, ..., A_{3^k}, B_1, .., B_{3^k}$ and $C > 0$ such that 
\begin{align*}
\left| \frac{d^k}{dt^k} F_{s}(t)\right| & \leq \sum_{i=1}^{3^k} C_i \frac{|t|^{c_i} (\arccos t )^{-s_i}}{(1-t^2)^{b_i}}\\
& \leq \sum_{i=1}^{3^k} B_i ( 1-t)^{- (b_i + \frac{s_i}{2})} (1+t)^{- b_i}\\
& \leq \sum_{i=1}^{3^k} A_i ( 1-t)^{-k - \frac{s}{2}} (1+t)^{- k + \frac{1}{2}}\\
& \leq C ( 1-t)^{-k  - \frac{s}{2}} (1+t)^{- k + \frac{1}{2}} \numberthis \label{eq:Fa IBP derivative bound}
\end{align*}
on $(-1,1)$. Recall that the condition that $s>-1$ guarantees that all $s_i \ge s+1$ are positive. 

Applying \eqref{eq:Fa IBP Derivative of weights}, we have, for $0 \leq k \leq n-1$,
\begin{equation*}
\left|\left(\frac{d^k}{dt^k} F_s (t)\right)\left(\frac{d^{n-k-1}}{dt^{n-k-1}}\left((1-t)^{n+\alpha}(1+t)^{n+ \beta}\right)\right) \right| \leq \sum_{j=0}^{n-k-1} G_j (1-t)^{n + \alpha - j  - k - \frac{s}{2}}  (1+t)^{\beta + j + \frac{3}{2}},
\end{equation*}
proving \eqref{eq:Fa IBP endpoints zeros}, since  $\alpha = \frac{D}{2} - 1 > \frac{s}{2} - 1$ and $\beta \geq - \frac{1}{2}$. 

Similarly, for $0 \leq k \leq n$, we have
\begin{equation*}
\left|\left(\frac{d^k}{dt^k} F_s (t)\right)\left(\frac{d^{n-k}}{dt^{n-k}}\left((1-t)^{n+\alpha}(1+t)^{n+ \beta}\right)\right) \right| \leq \sum_{j=0}^{n-k} G_j (1-t)^{n + \alpha - j  - k - \frac{s}{2}}  (1+t)^{\beta + j + \frac{1}{2}},
\end{equation*}
and once again using the facts that  $\alpha = \frac{D}{2} - 1 > \frac{s}{2} - 1$ and $\beta \geq - \frac{1}{2}$, we find that \eqref{eq:Fa IBP Integrable claim} is indeed integrable.
\end{proof}

This allows us to compute the integral \eqref{eq:RodriguesIntegral} using  integration by parts, for $n \geq 1$,
\begin{align*}
\int_{-1}^1 F_s (t) \frac{(-1)^n}{2^n n!}\frac{d^n}{dt^n} \left((1-t)^{n + \alpha}(1+t)^{n + \beta}\right) dt &=  \frac{(-1)^{n-1}}{2^n n!}\int_{-1}^1 \left( \frac{d}{dt}F_s(t) \right) \frac{d^{n-1}}{dt^{n-1}}\left((1-t)^{n + \alpha}(1+t)^{n + \beta}\right) dt\\
&= \frac{1}{2^n n!}\int_{-1}^1 \left(\frac{d^n}{dt^n} F_s(t)\right) (1-t)^{n + \alpha}(1+t)^{n + \beta} dt\\
&= \frac{1}{2^n n!}\int_{-1}^1 \sum_{i=1}^{3^n} C_i F_{s_i,b_i,c_i}(t) (1-t)^{n + \alpha}(1+t)^{n + \beta} dt\\
&=\frac{1}{2^n n!} \sum_{i=1}^{3^n} C_i \int_{-1}^1  \frac{t^{c_i} (1-t^2)^{n-b_i + \beta}(1-t)^{\alpha - \beta}}{(\arccos t)^{s_i}} dt\\
\end{align*}
We now show that for all $i= 1, \dots, 3^k$, the integral 
$$\int_{-1}^1   \frac{t^{c_i} (1-t^2)^{n-b_i + \beta}(1-t)^{\alpha - \beta}}{(\arccos t)^{s_i}} dt$$
is positive. If $c_i$ is even, then the integrand is positive on $(-1,1) \setminus \{0\}$, hence the integral is positive. Otherwise, we rewrite the integral as
\begin{equation*}
\int_{-1}^1  \frac{t^{c_i} (1-t^2)^{n-b_i + \beta}(1-t)^{\alpha - \beta}}{(\arccos t)^{s_i}} dt = \int_{0}^1  t^{c_i} (1-t^2)^{n-b_i + \beta}
\left(\frac{(1-t)^{\alpha - \beta}}{(\arccos t)^{s_i}}-\frac{(1+t)^{\alpha - \beta}}{(\arccos (-t))^{s_i}}\right) dt
\end{equation*}
and show that the integrand is positive. 

\begin{proposition}\label{prop:Positivitiy of odd terms}
For $s_i \ge 2(\alpha - \beta)$ the expression
\begin{equation*}
p(t) \coloneqq \frac{(1-t)^{\alpha - \beta}}{(\arccos t)^{s_i}}-\frac{(1+t)^{\alpha - \beta}}{(\arccos(-t))^{s_i}}
\end{equation*}
is positive on $(0,1)$. 
\end{proposition}

\begin{proof}
Since $p(0) = 0$, it suffices to show that
\begin{align*}
\frac{d}{dt} p(t)&= (\arccos t)^{-s_i -1} (1-t)^{\alpha - \beta - \frac{1}{2}} \left( \frac{s_i}{\sqrt{1+t}} - \frac{(\alpha - \beta) \arccos t}{\sqrt{1-t}} \right)\\
& \qquad + (\arccos (-t))^{-s_i -1} (1+t)^{\alpha - \beta - \frac{1}{2}} \left( \frac{s_i}{\sqrt{1-t}} - \frac{(\alpha - \beta) \arccos (-t)}{\sqrt{1+t}} \right)
\end{align*}
is positive on $(0,1)$. We have this immediately when $\alpha = \beta$ (when $\Omega$ is a sphere or $\mathbb{FP}^1$), so we now consider the cases where $\alpha > \beta$. We see that this positivity occurs when 
\begin{equation}\label{eq:Positivity aux}
(\alpha - \beta)\arccos t \cdot \sqrt{\frac{1+t}{1-t}} <  s_i
\end{equation}
on $(-1,1)$. Using the elementary inequality $\arccos t >  \sqrt{1-t^2}$ for $t\in (-1,1)$, we see that 
\begin{align*}
\frac{d}{dt}\left((\alpha - \beta)\arccos t \cdot \sqrt{\frac{1+t}{1-t}}\right) &= (\alpha - \beta) \frac{t^2+\arccos t \cdot  \sqrt{1-t^2}-1}{(1-t)^2(1+t)}\\
&> (\alpha - \beta) \frac{t^2+(1-t^2)-1}{(1-t)^2(1+t)} = 0,
\end{align*}
on $(-1,1)$, i.e. the expression on the left-hand side of \eqref{eq:Positivity aux} is strictly increasing.  Since
$$ \lim_{t\to1^-}  (\alpha- \beta)\arccos t \cdot \sqrt{\frac{1+t}{1-t}} =  2 (\alpha - \beta),$$ 
the strict inequality \eqref{eq:Positivity aux} holds for  $t \in [-1,1)$ if $s_i \geq 2(\alpha - \beta)$.
\end{proof}


Since according to Proposition \ref{prop:FDerivatives}, $s_i \geq s+1$,  Proposition \ref{prop:Positivitiy of odd terms} holds when $s \geq 2(\alpha - \beta) - 1$.
Thus,
\begin{equation*}
\int_{-1}^1 F_{s}(t) P_n^{(\alpha, \beta)}(t) (1-t)^{\alpha} (1+t)^{\beta} dt > 0
\end{equation*}
for all $n \geq 1$ when $ 2(\alpha - \beta) - 1 \leq s < D$, as long as $s>-1$, which finishes the proof of Theorem \ref{thm:Pos Jacobi Coefficients Geod Riesz}.

\section{Energy minimization with singular kernels}\label{s.singularkernels}

Now we want to show that if the Jacobi coefficients of $f$ are nonnegative  and $f$ has certain other additional properties, then  $I_f (\mu)$ is minimized by the uniform measure $\sigma$.  If the kernel $f$ is continuous, one could simply invoke Theorem \ref{thm:uniform measure min, continuous} in order to make this conclusion (this would apply to the case $-1<s<0$).

However, for our purposes, the most interesting  kernels are singular ($s\geq 0$), and Theorem \ref{thm:uniform measure min, continuous} does not apply directly.  We shall use different  approximation arguments to first  prove this statement for all absolutely continuous (with respect to $\sigma$) measures and then extend it to all Borel measures on $\Omega$. In what follows, we assume that the Jacobi coefficients of $f$ are nonnegative.

\subsection{The proof for absolutely continuous measures: Ces\`aro summation}\label{s.abs}
To construct a continuous approximation of the kernel, we shall employ Ces\`aro means. Assume that  $f \in  L^1 \big( [-1,1],    \nu^{(\alpha, \beta)} \big)$, i.e. the function $f (\cos(\vartheta(x,y)))$ is integrable on $\Omega$ with respect to $\sigma$ in both $x$ and $y$.  For $\delta \ge 0$ and $k \in \mathbb{N}_0$, define
\begin{equation}
A_k^\delta = {{k + \delta} \choose {k}} = \frac{(\delta +k) (\delta +k -1) \dots (\delta +1)}{k!}.
\end{equation}
Given the Jacobi expansion of $f$, as defined in \eqref{eq:Jacobi Expansion}, the Ces\`{a}ro $(C,\delta)$ means of $f$ is the polynomial
\begin{equation}\label{eq:Cesaro}
S_n^\delta f (t) = \frac{1}{A_n^\delta} \sum_{k=0}^n A_{n-k}^\delta  \widehat{f}_k  P_k^{(\alpha, \beta)} (t). 
\end{equation}
It is known that for $\delta > \alpha + \frac{1}{2}$, $\alpha \geq \beta > -1$,  the Ces\`aro means $S_n^\delta f$ converge to $f$ in the norm of $L^1 \big( [-1,1],  \nu^{(\alpha, \beta)} \big)$, see \cite[Theorem 2.5.3]{DaiX}, or, equivalently, 
\begin{equation}
\int_{\Omega} \int_{\Omega}  \big| S_n^\delta f (\cos(\vartheta(x,y))) - f (\cos(\vartheta(x,y)))  \big|\, d\sigma (x) d\sigma (y) \rightarrow 0 \,\,\, \textup{ as } \,\,\,\, n\rightarrow \infty. 
\end{equation}
The basic idea of the proof  is as follows: it is easy to see that for all $p \in [1, \infty]$,  $\lim\limits_{n \rightarrow \infty} \|S_n^{\delta} f - f \|_p = 0$ for any polynomial. Furthermore,    \cite[Corollary 18.11]{ChaM93} shows that if $1 \leq p \leq \infty$, $\delta > \alpha + \frac{1}{2}$, $\alpha \geq \beta > -1$ and $f \in L^p([-1,1], \nu^{(\alpha, \beta)})$, then $\| S_n^\delta f\|_{p} \leq c_p \|f\|_p$, i.e.  Ces\`{a}ro means form a  uniformly bounded family of operators on $L^p$.    
This implies $L^p$ convergence by a standard density argument. 



It is easy to see that if the Jacobi coefficients $\widehat{f}_n$ of $f$ are positive, the same is true for $S_n^\delta f$. Since the function $S_n^\delta f $ is continuous (it is a polynomial), Theorem \ref{thm:uniform measure min, continuous} applies to $S_n^\delta f$, i.e. the energy $I_{S_n^\delta f} (\mu)$ is minimized by $\sigma$.

We first assume that a Borel probability  measure $\mu$ on $\Omega$ is absolutely continuous with respect to $\sigma$ and has bounded density, i.e. $d\mu (x) = h (x) d\sigma (x)$ for some $h \in L^{\infty}(\Omega, \sigma)$, with $h$ not  identically $1$, so that $\mu \neq \sigma$. This then implies that there is some $m \in \mathbb{N}$ such that $I_{P_m^{(\alpha, \beta)}}(\mu) > I_{P_m^{(\alpha, \beta)}}(\sigma) = 0$ (see Lemma \ref{l.App1} in the Appendix),  and for all other $n \in \mathbb{N}$, $I_{P_n^{(\alpha, \beta)}}(\mu) \geq I_{P_n^{(\alpha, \beta)}}(\sigma) = 0$ due to \eqref{eq:Jacobi Pos def}. 
Thus for $n \geq m$,
\begin{align*}
I_{S_n^\delta f} (\mu) = \frac{1}{A_n^{\delta}} \sum_{k=0}^n A_{n-k}^\delta  \widehat{f}_k  I_{P_k^{(\alpha, \beta)}}(\mu) & \geq \widehat{f}_0 + \frac{A_{n-m}^\delta }{A_n^{\delta}} \widehat{f}_m  I_{P_m^{(\alpha, \beta)}}(\mu)\\
& = I_{S_n^\delta f} (\sigma) + \frac{A_{n-m}^\delta }{A_n^{\delta}} \widehat{f}_m  I_{P_m^{(\alpha, \beta)}}(\mu).
\end{align*}

Then for any $\varepsilon >0$, by choosing $n \geq m$ large enough,  we obtain 
\begin{align*}
I_{f} (\mu) & =  \int_{\Omega} \int_{\Omega} f  (\cos(\vartheta(x,y)))  \, d\mu (x) d\mu (y) \\
& \geq  I_{S_n^\delta f} (\mu) -   \int_{\Omega} \int_{\Omega} \big| S_n^\delta f  (\cos(\vartheta(x,y))) - f  (\cos(\vartheta(x,y)))  \big|\, |h(x)| \, |h(y) | \, d\sigma (x) d\sigma (y)\\
& \ge I_{S_n^\delta f} (\sigma) + \frac{A_{n-m}^{\delta}}{A_n^{\delta}} \widehat{f}_m I_{P_m^{(\alpha, \beta)}}(\mu) - \|h\|_\infty^2 \int_{\Omega} \int_{\Omega}  \big| S_n^\delta f (\cos(\vartheta(x,y))) - f (\cos(\vartheta(x,y)))  \big|\, d\sigma (x) d\sigma (y) \\
& \ge I_{S_n^\delta f} (\sigma) + \frac{A_{n-m}^{\delta}}{A_n^{\delta}} \widehat{f}_m I_{P_m^{(\alpha, \beta)}}(\mu) - \frac{\varepsilon}{2} \ge I_{f} (\sigma) + \frac{A_{n-m}^{\delta}}{A_n^{\delta}} \widehat{f}_m I_{P_m^{(\alpha, \beta)}}(\mu)  - \varepsilon. 
\end{align*}

Since $\varepsilon >0$ is arbitrary, we find that
\begin{align*}
I_{f} (\mu) & \geq I_{f}(\sigma) + \lim_{n \rightarrow \infty} \frac{A_{n-m}^{\delta}}{A_n^{\delta}} \widehat{f}_m I_{P_m^{(\alpha, \beta)}}(\mu)\\
& = I_{f}(\sigma) + \lim_{n \rightarrow \infty} \frac{n (n-1) \cdots (n-m+1)}{(n + \delta) \cdots (n-m + \delta)} \widehat{f}_m I_{P_m^{(\alpha, \beta)}}(\mu) \\
& = I_{f}(\sigma) + \widehat{f}_m I_{P_m^{(\alpha, \beta)}}(\mu)\\
 & \geq I_{f}(\sigma),
\end{align*}
with the last inequality being strict if $\widehat{f}_m > 0$. This gives us the following result.
 
\begin{proposition}\label{prop:sigma min abs cont measures bound density}
Suppose $f \in L^1( [-1,1], \nu^{(\alpha, \beta)})$. Then the following hold:
\begin{itemize}
\item If for $n \in \mathbb{N}$, $\widehat{f}_n \geq 0$, then $\sigma$ is a minimizer of $I_f$ among all probability measures of the form $d\mu(x) = h(x) d\sigma(x)$ for some $h \in L^{\infty}(\Omega, \sigma)$.
\item If for $n \in \mathbb{N}$, $\widehat{f}_n > 0$, then $\sigma$ is the {unique minimizer} of $I_f$ among all probability measures of the form $d\mu(x) = h(x) d\sigma(x)$ for some $h \in L^{\infty}(\Omega, \sigma)$.
\end{itemize} 
\end{proposition}

Next, we assume that $\mu$ is any Borel probability measure  absolutely continuous with respect to $\sigma$, i.e. $d\mu (x) = h(x) d\sigma (x)$ with $h\ge 0$, $h \in L^1 (\Omega, \sigma)$,  and $\| h \|_1 =1 $.  Assume that $I_f (\mu) < \infty$ (otherwise there is nothing to prove), define $h_n = \min \{ h,n \}$, and define a measure $\mu_n$ by setting $d\mu_n (x) = h_n (x) d\sigma (x)$. Notice that $\mu_n$ is not a probability measure, but $\| \mu_n \| = \| h_n \|_{1} \rightarrow 1 $ as $n\rightarrow \infty$ by the Lebesgue dominated convergence theorem.  Therefore, if $f\ge 0$, 
\begin{align*}
I_f (\mu) \ge I_f (\mu_n) \ge \| h_n \|_{1}^2   I_f (\sigma)  \rightarrow   I_f (\sigma) \,\,\, \textup{ as } \,\,\,\, n\rightarrow \infty,
\end{align*}
If $f$ is not positive, but is bounded from below, the same conclusion follows simply by adding a constant to $f$. This leads to the following statement. 

\begin{proposition}
Suppose that $f \in L^1( [-1,1], \nu^{(\alpha, \beta)})$ is bounded from below. If for $n \in \mathbb{N}$, $\widehat{f}_n \geq 0$, then $\sigma$ is a minimizer of $I_f$ among all probability measures of the form $d\mu(x) = h(x) d\sigma(x)$ for some $h \in L^{1}(\Omega, \sigma)$.
\end{proposition}

It is not straightforward to obtain a uniqueness statement akin to the second part of Proposition \ref{prop:sigma min abs cont measures bound density}, but we shall circumvent it later on (bounded density will suffice in the approximation arguments of the next subsection). 



\subsection{Arbitrary measures: averaging approximations}\label{s.all}

In order to pass to all measures, we shall approximate both the kernel and the measure by averaging them over balls of small radius.  We begin with some auxiliary definitions. 

Let $B_\varepsilon (x)$ denote a ball  on $\Omega$ centered at $x \in \Omega$ and of geodesic radius $\varepsilon > 0$, i.e. $ B_\varepsilon(x) = \{ z \in \Omega :  \vartheta(z,x) < \varepsilon\}$.  As the uniform surface measure of such a ball is independent of the center, we shall denote it simply by $\sigma (B_\varepsilon)$. 

Consider  a non-negative function  $f \in L^1 \big( [-1,1],    \nu^{(\alpha,\beta)} \big)$. Since $f ( \cos(\vartheta(x,y)))$ is integrable with respect to $d\sigma (x) d\sigma (y)$, we can define an approximation

\begin{equation}\label{eq:fepsilon}
f^{(\varepsilon)} (\cos(\vartheta(x,y)) ) =  \frac{1}{\big(  \sigma (B_\varepsilon) \big)^2 } \int_{B_\varepsilon(x)} \int_{B_\varepsilon (y )} f (\cos(\vartheta(u,v)) ) d\sigma (u) d\sigma (v). 
\end{equation}
It is obvious that  $f^{(\varepsilon)}$ depends just on the quantity $\cos(\vartheta(x,y))$,  $f^{(\varepsilon)}$  is continuous (and hence bounded),  and that $f^{(\varepsilon)} (\cos(\vartheta(x,y))) \rightarrow f (\cos(\vartheta(x,y)))$  as $\varepsilon \rightarrow 0^+$ for all $x,y \in \Omega$ such that $f $ is continuous at $\cos(\vartheta(x,y))$ (in the extended sense). This convergence also holds $\sigma\times \sigma$-a.e. due to the Lebesgue differentiation theorem, but this is not enough for our purposes since we shall need to consider different measures, most importantly, those not absolutely continuous with respect to $\sigma$. We would like to remark that similar averages for the logarithmic kernel on $\mathbb S^2$ and Green kernels on $\mathbb{S}^d$ have been used for a different energy problem (and even computed explicitly)  in \cite{BelL}.   \\

Similarly, given a Borel probability measure $\mu$ on $\Omega$, we define its approximation by 
\begin{equation}
d \mu^{(\varepsilon)}  (x) = \frac{\mu \big(B_\varepsilon(x) \big)}{\sigma  \big(B_\varepsilon(x) \big)} d\sigma (x) =  \bigg( \frac{1}{  \sigma (B_\varepsilon) } \int_{B_\varepsilon(x)} d\mu (u) \bigg) d\sigma (x).
\end{equation}
Notice that $\mu^{(\varepsilon)}$ is a probability measure since 
\begin{equation}
\int_{\Omega} d\mu^{(\varepsilon)} (x) =   \frac{1}{  \sigma (B_\varepsilon) }  \int_{\Omega} \int_{B_\varepsilon(x)} d\mu (u) d\sigma (x) 
=   \frac{1}{  \sigma (B_\varepsilon) }  \int_{\Omega} \bigg( \int_{B_\varepsilon(u)}  d\sigma (x)  \bigg) d\mu (u)  =  \int_{\Omega}  d\mu (u ) = 1,
\end{equation}
and this also shows that this  measure is absolutely continuous with respect to $\sigma$, with bounded density (using the convention $d\mu(x) = h(x) d\sigma(x)$)
$$h(x) = 
 \frac{\mu \big(B_\varepsilon(x) \big)}{\sigma  \big(B_\varepsilon(x) \big)}  \leq  \frac{1}{  \sigma (B_\varepsilon) }.$$

\noindent We also observe that the following relation holds between the approximations $\mu^{(\varepsilon)}$ and $f^{(\varepsilon)}$:
\begin{equation} \label{eq:muandf}
I_f \big( \mu^{(\varepsilon)} \big)   =  I_{f^{(\varepsilon)}} (\mu). 
\end{equation}
Indeed, 
\begin{align}
\nonumber I_f \big( \mu^{(\varepsilon)} \big) & = \int_{\Omega} \int_{\Omega}  f (\cos(\vartheta(x,y)) ) d\mu^{(\varepsilon)} (x) d\mu^{(\varepsilon)}(y) \\
\nonumber  & =  \frac{1}{\big(  \sigma (B_\varepsilon) \big)^2 }  \int_{\Omega} \int_{\Omega} \int_{B_\varepsilon(x)} \int_{B_\varepsilon(y)}   f (\cos(\vartheta(x,y)) ) \, d\mu (u) d\mu (v) d\sigma (x) d\sigma (y) \\
\nonumber & =  \int_{\Omega} \int_{\Omega} \bigg(  \frac{1}{\big(  \sigma (B_\varepsilon) \big)^2 }    \int_{B_\varepsilon(u)} \int_{B_\varepsilon(v)}   f( \cos(\vartheta(x,y)))  d\sigma (x) d\sigma (y) \bigg)  \, d\mu (u) d\mu (v)\\
\nonumber & =  \int_{\Omega} \int_{\Omega}  f^{(\varepsilon)} (\cos(\vartheta(u,v)) ) d\mu (u) d\mu (v)  =  I_{f^{(\varepsilon)}} (\mu).  
\end{align}
We note that this all holds true if $\mu$ is a finite signed Borel measure. In this case $\mu^{(\varepsilon)}$ is a finite signed measure, with $\mu^{(\varepsilon)}(\Omega) = \mu(\Omega)$, and still an absolutely continuous measure with bounded density.

We now show that nonnegativity of Jacobi coefficients is preserved by the approximation $f^{(\varepsilon)}$.
\begin{proposition}\label{prop:epsilon approx function pos def}
Suppose $f \in L^1( [-1,1], \nu^{(\alpha, \beta)})$ and let $ \varepsilon>0$. 
If for $n \in \mathbb{N}$, $\widehat{f}_n \geq 0$, then the continuous kernel $f^{(\varepsilon)}( \cos(\vartheta(x,y)))$ is conditionally positive definite, i.e.  $\widehat{f^{(\varepsilon)}}_n \geq 0$ for all $n\in \mathbb N$. 
\end{proposition}

\begin{proof}
By Proposition \ref{prop:sigma min abs cont measures bound density}, we have that $\sigma$  minimizes $I_f$ over all absolutely continuous probability measures with bounded density. Applying \eqref{eq:muandf} and observing that $\sigma^{(\varepsilon)} = \sigma$, we find that
\begin{equation*}
I_{f^{(\varepsilon)}}(\mu) = I_f( \mu^{(\varepsilon)}) \geq I_{f}(\sigma) = I_{f}(\sigma^{(\varepsilon)}) = I_{f^{(\varepsilon)}}(\sigma).
\end{equation*}
Since $f^{(\varepsilon)}$ is continuous, Theorem \ref{thm:uniform measure min, continuous} then implies that the function is conditionally  positive definite and has nonnegative  Jacobi coefficients, aside from the constant term.
\end{proof} 

\begin{remark} It is interesting to observe that  the strict positive definiteness statement in this setting is somewhat delicate. For some $\mu \neq \sigma$ and some $\varepsilon >0$, it might happen that $\mu^{(\varepsilon)} = \sigma$. In fact, on the sphere $\mathbb S^d$, the so-called `freak theorem' \cite{freak1,freak2} states that for a countable dense (and very specific) set of values of $\varepsilon>0$, there exists a non-zero continuous  function $f$ on $\mathbb S^d$, such that its average over every ball $B_\varepsilon$ is zero. Then taking $d \mu =(1+ c f) d\sigma$  would produce a measure such that $\mu^{(\varepsilon)} = \sigma$. Thus the exact form and the proof of  strict positive definiteness in  Proposition \ref{prop:epsilon approx function pos def} would require extensions of the `freak theorem' to projective spaces, which are most likely possible but are beyond the scope of this paper. However, we shall be able to avoid this issue while proving uniqueness of  minimizers. 
\end{remark}


We finally state a fairly general theorem which guarantees that positivity of Jacobi coefficients of a singular kernel  implies that the uniform measure $\sigma$ minimizes the energy integral.

\begin{theorem}\label{thm:sigma a minimizer for epsilon convergence}
Let $f \in L^1([-1,1], \nu^{(\alpha, \beta)})$ be bounded from below, and assume that for all $\mu \in \mathcal{P}(\Omega)$ such that $I_f(\mu) < \infty$,
\begin{equation}\label{eq:converge}
\lim_{\varepsilon \rightarrow 0^+} I_{f^{(\varepsilon)}}(\mu) = I_f(\mu).
\end{equation}
Then the following statements hold
\begin{itemize}
\item If for $n \in \mathbb{N}$, $\widehat{f}_n \geq 0$, then $\sigma$ is a minimizer of $I_{f}$ over all probability measures on $\Omega$.
\item If for $n \in \mathbb{N}$, $\widehat{f}_n > 0$, then $\sigma$ is the unique minimizer of $I_f$ over all probability measures on $\Omega$. 
\end{itemize}
\end{theorem}

\begin{proof}
By \eqref{eq:Equivalence with 1 dim setting}, we see that $I_f(\sigma) < \infty$. 
Lemma \ref{l.App2} in the Appendix shows that the convergence of energy integrals in the assumption of the theorem guarantees the convergence of  Jacobi coefficients, i.e.   for any $n \in \mathbb{N}$, we have
 $\lim\limits_{\varepsilon \rightarrow 0} \widehat{f^{(\varepsilon)}}_n = \widehat{f}_n$. 

Let $\mu \in \mathcal{P}(\Omega)$ such that $\mu \neq \sigma$, $I_f(\mu) < \infty$, i.e $f( \cos(\vartheta(x,y)))$ is integrable with respect to $\mu \times \mu$.   Since the Jacobi polynomials are positive definite, for all  $n \in \mathbb{N}$, we must have $I_{P_n^{(\alpha, \beta)}}(\mu) \geq 0$, and since $\mu \neq \sigma$, there must exist some $m \in \mathbb{N}$ such that $I_{P_m^{(\alpha, \beta)}}(\mu) > 0 = I_{P_m^{(\alpha, \beta)}}(\sigma)$, see Lemma \ref{l.App1} in the Appendix. Combining this with Proposition \ref{prop:epsilon approx function pos def} and the fact that $I_{f^{(\varepsilon)}}(\sigma) = I_f(\sigma) = \widehat{f}_0$, we have that
\begin{align*}
I_f(\mu) - I_f(\sigma) & = \lim_{\varepsilon \rightarrow 0^+} \Big( I_{f^{(\varepsilon)}}(\mu) - I_{f^{(\varepsilon)}}(\sigma) \Big) \\
 & = \lim_{\varepsilon \rightarrow 0^+} \sum_{n=1}^{\infty} \widehat{f^{(\varepsilon)}}_n  \int_{\Omega} \int_{\Omega} P_n^{(\alpha, \beta)}(\cos(\vartheta(x,y))) d\mu(x) d\mu(y) \\
& \geq \lim_{\varepsilon \rightarrow 0^+} \widehat{f^{(\varepsilon)}}_m  \int_{\Omega} \int_{\Omega} P_m^{(\alpha, \beta)}(\cos(\vartheta(x,y))) d\mu(x) d\mu(y) \\
& = \widehat{f}_m \int_{\Omega} \int_{\Omega} P_m^{(\alpha, \beta)}(\cos(\vartheta(x,y))) d\mu(x) d\mu(y).
\end{align*}
Here, in order to interchange the order of the integration and summation, we have used the fact that the Jacobi series of positive definite continuous kernels converge absolutely, see e.g. \cite[Lemma 2.16]{AndDGMS}, 
which can be viewed as a consequence of Mercer's Theorem \cite{Mer} (see \cite[Theorem 2.18]{BilMV}  for more general spaces). 

Thus, if $\widehat{f}_n \geq 0$ for all $n \in \mathbb{N}$ then $I_f(\mu) \geq I_f(\sigma)$ and if $\widehat{f}_n > 0$ for all $n \in \mathbb{N}$ then $I_f(\mu) > I_f(\sigma)$, and the theorem is proved. 
\end{proof}

 Theorem \ref{thm:sigma a minimizer for epsilon convergence} implies that for singular  kernels with positive Jacobi coefficients  the goal of understanding  when the energy is minimized by the uniform measure $\sigma$ is reduced to   checking  condition \eqref{eq:converge}, i.e. that  $I_{f^{(\varepsilon)}} (\mu) \rightarrow I_f (\mu)$ as $\varepsilon \rightarrow 0^+$.  In the following subsection we address this issue for the geodesic Riesz energies, and in Section \ref{s:further} we discuss extensions to more general kernels.  \\

\subsection{$A_1$ inequalities and the proof of Theorem \ref{t.geodriesz}.}\label{s:A1proof}

Let us return to the function under consideration $F_s (t ) = \frac{1}{s} (\arccos t )^{-s}$ for $s>0$ and $F_0(t) = - \log( \arccos(t))$. We shall prove the following proposition:

\begin{proposition}\label{p.max}
Let $0 \leq s <D $. There exists constants $A, C>0$ such that for all $\varepsilon >0$,  and for all $ x, y \in \Omega$ 
we have 
\begin{equation}\label{e.max1}
F_s^{(\varepsilon )} ( \cos(\vartheta(x,y))) \le \begin{cases}
 C F_s (\cos(\vartheta(x,y))), & s > 0\\
 A + C F_0 (\cos(\vartheta(x,y))), & s = 0
 \end{cases}
\end{equation}
or, in other words, 
\begin{align}
\label{e.max2} \sup_{\varepsilon > 0} \frac{1}{\big(  \sigma (B_\varepsilon) \big)^2 } \int_{B_\varepsilon(x)} \int_{B_\varepsilon (y )}  (\vartheta(u,v) )^{-s}  d\sigma (u) d\sigma (v)  & \le  C (\vartheta(x,y))^{-s}, & s > 0\\
\label{e.maxlog} \sup_{\varepsilon > 0} \frac{1}{\big(  \sigma (B_\varepsilon) \big)^2 } \int_{B_\varepsilon(x)} \int_{B_\varepsilon (y )} \big( - \log (\vartheta(u,v)) \big) d\sigma (u) d\sigma (v) & \le 
 A  -  C \log (\vartheta(x,y)). & { }
\end{align}
\end{proposition}


\begin{remark}
Observe that the left-hand side of \eqref{e.max2} is the bivariate analogue of the classical Hardy--Littlewood maximal function and the inequality states that $M_{HL} (F_s) \le CF_s$ pointwise. This inequality  essentially coincides with the definition of the so-called Muckenhaupt $A_1$ weights (see, e.g., \cite[Chapter 7]{Gra}) and is thus consistent with the fact that the function $w(x) = \|x\|^\alpha $ (i.e. the Euclidean distance) on $\mathbb R^d$ is an $A_1$ weight precisely when $-d < \alpha <0$ (see Example 7.1.7 in \cite{Gra}, see also Example 7.1.8 for the logarithmic kernel on $\mathbb R^d$). See further discussion in \S\ref{s:further}. 
\end{remark}

We also note that \eqref{e.maxlog} looks slightly differently since $F_0$  can take negative values. By slightly redefining the logarithmic kernel, we could ensure that this inequality also has the form $F_0^{(\varepsilon)} \le C F_0$, but don't do it in order to keep the definition of the kernel $F_0$ consistent with the potential theory literature. \\

Assuming Proposition \ref{p.max} (whose proof is postponed to \S\ref{s:proofA1} below), we are now equipped to finish the proof of Theorem \ref{t.geodriesz} for $s\geq 0$.  Fix a probability measure $\mu$ such that $I_{F_s } (\mu)  < \infty$, i.e. $F_s(\cos(\vartheta(x,y)))$ is integrable with respect to $\mu \times \mu$  (otherwise, there is nothing to prove).  This, in particular, implies that the measure $\mu$ is atomless.  This easily implies that $\mu \times \mu ( \{ x =  y\} ) = 0$.  Since  $F^{(\varepsilon)}_s ( \cos(\vartheta(x,y)) ) \rightarrow F_s  (\cos(\vartheta(x,y)))$  as $\varepsilon \rightarrow 0^+$ for all $x,y \in  \Omega$ with $x\neq  y$, we have $\mu \times \mu$-a.e. convergence. Inequality \eqref{e.max1} of Proposition \ref{p.max}  then allows us to invoke the Lebesgue dominated convergence   theorem to obtain that  
\begin{equation}
I_{F_s^{(\varepsilon)} }( \mu ) \rightarrow I_{F_s} (\mu)\,\,\,  \textup{ as } \,\, \varepsilon \rightarrow 0^+.
\end{equation} 
Since in Theorem \ref{thm:Pos Jacobi Coefficients Geod Riesz} we have proved positivity of Jacobi coefficients of $F_s$  in the ranges indicated in Theorem \ref{t.geodriesz}, as well as non-negativity of such coefficients when $\Omega = \mathbb S^d$ or $\mathbb{FP}^1$ and $s=-1$, we can just apply Theorem \ref{thm:sigma a minimizer for epsilon convergence} to immediately prove Theorem  \ref{t.geodriesz}.  This in turn implies Theorems \ref{t.main} and \ref{t.d1}.


\subsubsection{Proof of the maximal inequality \eqref{e.max1}}\label{s:proofA1}
It remains to prove Proposition \ref{p.max}. We assume that $\varepsilon$ is small, say $\varepsilon < \frac{\pi}{8}$, since for large caps the averages in \eqref{e.max2} are clearly bounded by a constant.



We start with two lemmas which deal with estimates of the averages in just one variable. 
In this section we shall employ the notation $X \lesssim Y$ to mean that $ X \le C Y$ for some constant $C>0$. The implicit constant $C$ may depend on $d$, $\alpha$, $\beta$, and $s$, but stays independent of $x$, $y$, and $\varepsilon$. 
\begin{lemma}\label{l.1max1} 
Under the same conditions as in Proposition \ref{p.max},  for all $s>0$ 
 there exists $C > 0$ such that for all $0< \varepsilon < \frac{\pi}{8}$
\begin{equation}\label{e.1max1}
\frac{1}{\sigma (B_\varepsilon) } \int_{B_\varepsilon (x)  } (\vartheta(u,y))^{-s}  d\sigma(u )\leq C\varepsilon^{-s}
\end{equation}
\end{lemma}

\begin{proof}
We first observe that  the expression $\int_{B_\varepsilon (x) } F_s (\cos(\vartheta(u,y))) d\sigma(u )$ is maximized when $x=y$.  Indeed, this can be seen as follows:
\begin{align*}
\int_{B_\varepsilon (x) }(\vartheta(u,y))^{-s}   d\sigma(u ) & \le \varepsilon^{-s} \sigma \big( B_\varepsilon (x) \setminus B_\varepsilon (y) \big) +   \int_{B_\varepsilon (x) \cap  B_\varepsilon (y)  } (\vartheta(u,y))^{-s}  d\sigma(u ) \\ 
& =  \varepsilon^{-s} \sigma \big( B_\varepsilon (y) \setminus B_\varepsilon (x) \big) +   \int_{B_\varepsilon (x) \cap  B_\varepsilon (y)  } (\vartheta(u,y))^{-s}   d\sigma(u )\\ & \le \int_{B_\varepsilon (y) } (\vartheta(u,y))^{-s}  d\sigma(u). 
\end{align*} 



We have observed before  that $\arccos (t) \ge \sqrt{2-2t}$  and thus $F_s (t ) \lesssim \big( 1-t \big)^{-\frac{s}{2}}$, 
since $s> 0$.  Applying \eqref{eq:Equivalence with 1 dim setting}, we see that 
\begin{align*}
\int_{B_\varepsilon (x)} (\vartheta(u,x))^{-s}  d\sigma(u) &\lesssim   \int_{\cos{\varepsilon}}^1 (1-t)^{-\frac{s}{2}+\alpha}  (1+t)^\beta \,dt 
\lesssim (1-\cos{\varepsilon})^{-\frac{ s}{2}+ \alpha + 1},
\end{align*}
where we used the fact that $s < D = 2\alpha + 2$, i.e. $-\frac{s}{2}+\alpha   > -1$.  Observe also that  $\beta = -1/2$ in the case of $\Omega=\RP^d$, hence in  this case we need to bound $(1+t)$ from below, which can be done since we only consider $\varepsilon < \frac{\pi}{8}$.

We also have
\begin{equation}\label{eq:Lower Bound on Ball Volume}
\sigma (B_\varepsilon) \gtrsim   \int_{\cos{\varepsilon}}^1 (1-t)^{\alpha} (1+t)^\beta dt  \gtrsim (1-\cos{\varepsilon})^{\alpha + 1 }.
\end{equation}
Combining these, along with the fact that $1-\cos{\varepsilon} \ge \frac{\varepsilon^2}{4}$ for $\varepsilon < \frac{\pi}{8}$, we obtain
\begin{align*}
\frac{1}{\sigma (B_\varepsilon) } \int_{B_\varepsilon (x)  }  (\vartheta(u,y))^{-s} d\sigma(u) & \leq  \frac{1}{\sigma (B_\varepsilon) } \int_{B_\varepsilon (x)  }  (\vartheta(u,x))^{-s}  d\sigma(u ) \\
&  \lesssim  \frac{ (1-\cos{\varepsilon})^{-\frac{s}{2}+\alpha +1 }}{(1-\cos{\varepsilon})^{\alpha + 1 }} = (1- \cos \varepsilon)^{-\frac{s}{2}} \lesssim \varepsilon^{-s}.
\end{align*}
which proves Lemma \ref{l.1max1}. 
\end{proof}

We prove an analogous statement for the logarithmic kernel. 

\begin{lemma}\label{l.1max1log} 
There exists constant $C, A > 0$ such that   for all $0< \varepsilon < \frac{\pi}{8}$ 
\begin{equation}\label{e.1max1log}
\frac{1}{\sigma (B_\varepsilon) } \int_{B_\varepsilon (x)  } F_0 (\cos(\vartheta(u,y))) d\sigma(u )\leq A - C \log \varepsilon
\end{equation}
\end{lemma}


\begin{proof} The proof is very similar. Just as in the proof of Lemma \ref{l.1max1} we note that the expression $ \int_{B_\varepsilon (x) } F_0 (\cos(\vartheta(u,y))) d\sigma(u )$ is maximized when  $x=y$. 
 
Again using that $\arccos t \ge \sqrt{2-2t}$  and therefore $F_0 (t ) \le -\frac{1}{2} \log(2-2t)$.  Applying \eqref{eq:Equivalence with 1 dim setting}, we see that
\begin{align*}
\int_{B_\varepsilon (x)} F_0(\cos(\vartheta(u,x)) d\sigma(u) &\lesssim  - \int_{\cos{\varepsilon}}^1 \big(\log 2  + \log(1-t) \big) (1-t)^{\alpha}  (1+t)^\beta \,dt \\ 
& \lesssim (1-\cos{\varepsilon})^{\alpha + 1} \Big( \frac{1}{\alpha + 1} - \log 2 - \log\big(1 - \cos \varepsilon  \big) \Big).
\end{align*}
Combining this with \eqref{eq:Lower Bound on Ball Volume} and the inequality $1-\cos{\varepsilon} \ge \frac{\varepsilon^2}{4}$, we obtain  
\begin{align*}
\frac{1}{\sigma (B_\varepsilon) } \int_{B_\varepsilon (x)  } F_0 (\cos(\vartheta(u,y))) d\sigma(u) 
& \lesssim  \frac{1}{\alpha + 1} - \log 2 - \log\big(1 - \cos \varepsilon \big) \\
&  \leq \frac{1}{\alpha + 1} + 3 \log 2 - 2 \log \varepsilon ,
\end{align*}
proving the lemma. 
\end{proof}

We now deduce a univariate version of Proposition \ref{p.max}.

\begin{lemma}\label{l.2max1} For all $0 <s< D$, 
there exists $C > 0$ such that for all $0< \varepsilon < \frac{\pi}{8}$
\begin{equation}\label{e.2max1}
\frac{1}{\sigma (B_\varepsilon) } \int_{B_\varepsilon (x)  }  (\vartheta(u,y))^{-s} d\sigma(u ) \leq C    (\vartheta(x,y))^{-s}
\end{equation}
\end{lemma}
\begin{proof}
Assume first that $\vartheta(x,y)  > 2\varepsilon$. Then for all $u \in B_\varepsilon (x)$, we have 
$\vartheta (u,y) >  \vartheta (x,y) - \varepsilon > \frac12 \vartheta(x,y)$,  and in this case we obtain 
 \begin{align*}
 \frac{1}{\sigma (B_\varepsilon) } \int_{B_\varepsilon (x)  }  (\vartheta(u,y))^{-s}  d\sigma(u ) &\leq  \frac{1}{\sigma (B_\varepsilon) } \int_{B_\varepsilon (x)  }\left(\frac{\vartheta(x,y) }{2}\right)^{-s} d\sigma(u) = 2^{s} \big( \vartheta (x,y) \big)^{-s}.
\end{align*}

Now assume that  $\vartheta(x,y) \le 2 \varepsilon$. Then Lemma \ref{l.1max1} guarantees that
\[
	\frac{1}{\sigma (B_\varepsilon) } \int_{B_\varepsilon (x)  }  (\vartheta(u,y))^{-s} d\sigma(u ) \leq C\, \varepsilon^{-s} \leq C\left(\frac{\vartheta(x,y)}{2}\right)^{-s} ={C} {2^{s}}  (\vartheta(x,y))^{-s}.
\]
We have thus proved  the lemma in both cases. \end{proof}

The logarithmic case is  similar.

\begin{lemma}\label{l.2max1log} 
There exist constant $C, A > 0$ such that  for all $0< \varepsilon < \frac{\pi}{8}$,
\begin{equation}\label{e.2max1log}
\frac{1}{\sigma (B_\varepsilon) } \int_{B_\varepsilon (x)  } F_0 (\cos(\vartheta(u,y))) d\sigma(u ) \leq A +C   F_0 (\cos(\vartheta(x,y)))
\end{equation}
\end{lemma}
\begin{proof}
For $\vartheta(x,y)  > 2\varepsilon$,  
 \begin{align*}
 \frac{1}{\sigma (B_\varepsilon) } \int_{B_\varepsilon (x)  } F_0 (\cos(\vartheta(u,y))) d\sigma(u ) &\leq  \frac{1}{\sigma (B_\varepsilon) } \int_{B_\varepsilon (x)  } - \log\left(\frac{\vartheta(x,y) }{2}\right) d\sigma(u)\\
&= \log 2 - \log(\vartheta(x,y))   = \log 2 + F_0 (\cos(\vartheta(x,y))),
\end{align*}
while for  $\vartheta(x,y) \le 2 \varepsilon$,  Lemma \ref{l.1max1log} guarantees that
\begin{align*}
\frac{1}{\sigma (B_\varepsilon) } \int_{B_\varepsilon (x)  } F_0 (\cos(\vartheta(u,y))) d\sigma(u )  & \leq A - C\log \varepsilon \leq A  - C \log\Big( \frac{\vartheta(x,y)}{2} \Big)\\
& = A + C \log 2 + C F_0(\cos(\vartheta(x,y))),
\end{align*}
and the claim  follows.
\end{proof}

Finally, we prove  Proposition \ref{p.max}. This simply amounts to applying Lemma \ref{l.2max1} or \ref{l.2max1log} twice in the cases that $s > 0$ or $s = 0$, respectively. In the former case, we see that
\begin{align*}
\frac{1}{\big(  \sigma (B_\varepsilon) \big)^2 }  \int_{B_\varepsilon(x)} \int_{B_\varepsilon (y )} &  (\vartheta(u,v))^{-s}  d\sigma (u) d\sigma (v)\\ 
&= \frac{1}{\sigma (B_\varepsilon) } \int_{B_\varepsilon (x)  }  \left(\frac{1}{\sigma (B_\varepsilon) } \int_{B_\varepsilon (y)  } (\vartheta(u,v))^{-s} d\sigma(v) \right)d\sigma(u) \\
&\leq \frac{1}{\sigma (B_\varepsilon) } \int_{B_\varepsilon (x)  } C \vartheta(u,y))^{-s} d\sigma(u) \leq C^2 (\vartheta(x,y))^{-s} ,
\end{align*}
and in the latter case, we have
\begin{align*}
\frac{1}{\big(  \sigma (B_\varepsilon) \big)^2 }  \int_{B_\varepsilon(x)} \int_{B_\varepsilon (y )} & F_0 (\cos(\vartheta(u,v)) )  d\sigma (u) d\sigma (v)\\ &= \frac{1}{\sigma (B_\varepsilon) } \int_{B_\varepsilon (x)  }  \left(\frac{1}{\sigma (B_\varepsilon) } \int_{B_\varepsilon (y)  } F_0 (\cos(\vartheta(u,v))) d\sigma(v) \right)d\sigma(u) \\
&\leq \frac{1}{\sigma (B_\varepsilon) } \int_{B_\varepsilon (x)  } \Big(A + C F_0 (\cos(\vartheta(u,y))) \Big) d\sigma(u) \\
& \leq A(1+C) + C^2 F_0 (\cos(\vartheta(x,y))),
\end{align*}
and thus Proposition \ref{p.max} is proved.

\subsection{Further discussion and generalizations}\label{s:further}
Since, as mentioned before, Theorem \ref{thm:uniform measure min, continuous} only applies to continuous kernels, additional considerations are required to show that positivity of the Jacobi coefficients implies that the uniform measure $\sigma$ minimizes the energy $I_f (\mu)$ for a singular function $f$. In all of the prior examples of this type found in the literature some ad hoc tricks have been used to this end, see the discussion after Theorem \ref{thm:uniform measure min, continuous}. To the best of our knowledge, there have been no general theorems which describe broad classes of singular kernels for which positivity of Jacobi coefficients  readily implies that $I_f (\mu)$ is minimized by $\sigma$. 
The arguments presented in  Sections \ref{s.abs}--\ref{s:A1proof} are rather flexible   and lend themselves  to generalizations, and we establish the first theorem of this type below, Theorem \ref{thm:general}, where the only condition imposed on $f$ is the order of its singularity.\\



First, we observe that the argument of \S\ref{s.abs} (absolutely continuous measures) applies to any integrable function which is bounded below, and so does the main statement of  \S\ref{s.all}, i.e. Theorem \ref{thm:sigma a minimizer for epsilon convergence}. 
In order to apply the latter  theorem, one needs to establish convergence in \eqref{eq:converge}. An important feature of the argument in \S\ref{s:A1proof}, allowing one to do so invoking the Lebesgue dominated convergence,   is the maximal inequality \eqref{e.max1} of Proposition \ref{p.max}, i.e. for the argument to hold, the kernel $f$ must satisfy the $A_1$ condition
\begin{equation}\label{e.A1}
\sup_{\varepsilon > 0} \frac{1}{\big(  \sigma (B_\varepsilon) \big)^2 } \int_{B_\varepsilon(x)} \int_{B_\varepsilon (y )} f ( \cos(\vartheta(u,v))) d\sigma (u) d\sigma (v) \le C f ( \cos(\vartheta(x,y)) ).
\end{equation}
If this is the case, as long as $f$ is continuous on $[-1,1)$, the previous argument can be repeated word for word.

Hence the question becomes: which kernels satisfy such an estimate? Of course, the proof of Proposition \ref{p.max}) would work almost verbatim, for example, for the powers of the chordal distance, but one can easily avoid having to prove this inequality from scratch. 

We  shall say that  non-negative functions $f$ and $g: [-1,1] \rightarrow \mathbb R$ are {\emph{equivalent}} on a set $K \subset [-1,1]$ if  for some $c'$, $c''>0$, the following inequality holds  
\begin{equation}\label{e.equiv}
c' g(t) \le f(t)  \le c'' g(t) \,\,\, \textup{ for all} \,\, t\in K. 
\end{equation}
It  is easy to see that if $f$ and $g$ are equivalent on $[-1,1]$, then whenever $f$ satisfies the $A_1$ condition \eqref{e.A1}, so does $g$  (with different constants). Since we have already established this condition for negative powers of the projective geodesic distance in Proposition \ref{p.max}, 
it must also hold for all kernels equivalent to these Riesz kernels. Moreover, it will suffice to establish the equivalence  just for   values of $t$ close to $1$. 

This leads to the following general statement for kernels which are singular at $t=1$, i.e. when $x=y$, in which the additional conditions on the singular kernel $f$ are fairly easy to check. 

\begin{theorem}\label{thm:general}
Assume that   a non-negative function  $f \in L^1([-1,1], \nu^{(\alpha, \beta)})$ is continuous for all $t\in [-1,1)$. Assume further that  the function $f$ is equivalent to  $(\arccos t)^{-s}$ for some $s \in (0,D)$ or to $ - \log ( \arccos t )$ on the interval $t\in [t_0,1]$  for some $t_0 \ge 3/4$ , in other words $f(\cos(\vartheta(x,y)))$ is equivalent to the negative $s^{th}$ powers of the geodesic distance $\vartheta(x,y)$ on $\Omega$ or to the negative logarithm of the distance whenever $\vartheta (x,y) \leq \arccos  t_0$.

If the Jacobi coefficients of $f$ are non-negative, i.e. $\widehat{f}_n \ge 0$ for all $n\ge 1$,  then the uniform measure $\sigma$ minimizes $I_f (\mu)$ among all Borel probability measures (and $\sigma$ is a unique minimizer if $\widehat{f}_n > 0$ for all $n\ge 1$).
\end{theorem}




\begin{proof}
Without loss of generality, we can assume that the equivalence holds on the whole interval $[-1,1]$. Indeed, 
if $f$ has zeros on $[-1,t_0]$, then  adding a positive constant to $f$ makes it bounded above and below on  $[-1,t_0]$ and thus equivalent to the Riesz kernel  on $[-1,1]$ (for $s=0$ the logarithmic kernel also needs to be changed by a constant to ensure positivity). Obviously adding a constant to the kernel doesn't affect energy minimizers. 

Due to Proposition \ref{p.max} and the equivalence, the function $f$ satisfies the $A_1$ inequality \eqref{e.A1}. Therefore, as in the proof of Proposition \ref{p.max}, we can  invoke the Lebesgue dominated convergence theorem to show that $I_{f^{(\varepsilon)}} (\mu) \rightarrow I_f (\mu)$ as $\varepsilon \rightarrow 0^+$ whenever $I_f (\mu ) < \infty$, and the statement follows from Theorem \ref{thm:sigma a minimizer for epsilon convergence}.
\end{proof}

Thus the only condition, in addition to the positivity of Jacobi coefficients, which is needed to ensure that the energy $I_f$ is minimized by the uniform measure $\sigma$, is the  order of singularity as $x$ approaches $y$ in terms of  the geodesic distance, or chordal/Euclidean distance, since these metrics are equivalent. We also note that in our approach we do need equivalence, i.e.  two-sided inequalities  in \eqref{e.equiv}.

Finally we would like to point out another well-known class of (singular) kernels for which it is known that the uniform measure $\sigma$ minimizes the energy on the sphere: kernels of the type $K(x,y) = f(\| x-y\|^2)$, where $f$ is strictly completely monotone, i.e. $(-1)^k f^{(k)} (t) > 0 $ for $k\in \mathbb N$ and $t>0$, and $f(0) = f (0+)$ (the statement follows from the fact that such kernels are positive definite, see Sections 4.4.2 and 4.4.3 of \cite{BorHS}). However, geodesic Riesz kernels don't fall into this category. Moreover, in Theorem \ref{thm:general} we do not assume any amount of smoothness of the kernel and rely only on different natural assumptions, i.e. Theorem \ref{thm:general} is clearly different and independent  from the aforementioned fact.


\section{The one-dimensional case}\label{s.d1}

In this section we shall further explore the one-dimensional case, i.e. the energy on $\mathbb{FP}^1$. We shall give an alternative proof of the case $d=1$ of Theorem \ref{t.geodriesz}, and consequently, of  Theorem  \ref{t.d1}. While these theorems have already been proved for all $d\ge 1$ in the  previous sections, we feel that the approach presented here better explains the nature of the case $d=1$. Moreover, we shall  obtain a better understanding of the minimizers in the case $s=-1$ and  discuss the original (discrete) Fejes T\'oth conjecture on $\mathbb S^1$.  The most important observation, on which the arguments of this section are based, is the fact that {\emph{the metric spaces $\mathbb S^D$ and $\mathbb{FP}^1$ with their respective geometric distances are isometric}}, where $D = \dim_{\mathbb{R}} \mathbb{F}$. Given that the geodesic distance Riesz energy on the sphere $\mathbb S^d$ is well understood in all dimensions $d\ge 1$, see \cite{BilD}, the results for the one-dimensional projective spaces follow immediately.\\


The close relationship between $\mathbb{FP}^1$ and $\mathbb{S}^D$ is well known (for example, \cite[p. 170]{Baez}), but it is not easy to find an explicitly stated example of an  isometry (most sources only discuss the homeo- or diffeomorphism). For the reader's convenience we shall briefly present it here. 

\vskip3mm

\vskip1cm 

\par
On the sphere we represent points $x\in \mathbb{S}^D$ by unit vectors in $\mathbb{R}^{D+1}$ and  express the geodesic distance on $\mathbb{S}^D$ as  
\[
\eta (x_1,x_2) = \arccos( x_1 \cdot x_2).
\]
On the projective line we represent points $\Pi \in \mathbb{FP}^1$ as $2\times 2$ projection matrices with entries in $\mathbb F$. 
A matrix $\Pi$ is a projection matrix if it is a Hermitian matrix which satisfies $\Pi^2 = \Pi$ and $\Tr \Pi = 1$. The geodesic distance on $\mathbb{FP}^1$ can then be expressed as
\[
\vartheta(\Pi_1,\Pi_2) = \arccos (2\langle \Pi_1,\Pi_2\rangle - 1), \,\,\,
\textup{ where } 
\,\,\,
\langle \Pi_1,\Pi_2 \rangle = \frac{1}{2} \Tr \left(\Pi_1\Pi_2 + \Pi_2\Pi_1\right) =  \Re \Tr \Pi_1\Pi_2.
\]
\begin{remark}\label{r.f2rep} When $\mathbb F \neq \mathbb O$, the projective line  $\mathbb {FP}^1$  can also be defined as the set of  equivalence classes of unit vectors in $\mathbb{F}^2$, where $(x,y) \sim (x',y')$ if $(x,y) = (x'\alpha,y'\alpha)$ for some $\alpha \in \mathbb{F}$ with $||\alpha|| = 1$. In this case, if $\Pi_1,\Pi_2$ are the projection matrices corresponding to $(x_1,y_1),(x_2,y_2)$, then \[
\langle \Pi_1,\Pi_2 \rangle = |\langle (x_1,y_1),(x_2,y_2)\rangle|^2 = |x_1\overbar{x_2}+y_1\overbar{y_2}|^2 .
\]
(for more details, see \cite{CohKM}). 
However, this construction fails when $\mathbb{F} = \mathbb{O}$, since we no longer have associativity. Hence projection matrices must be used for the general case.
\end{remark}

Following the exposition in  \cite{Baez}, all projection matrices can  be expressed as
\[
\Pi = \begin{pmatrix} \overbar{x}x & \overbar{x}y \\ \overbar{y}x & \overbar{y}y \end{pmatrix}
\]
where $x,y \in \mathbb F$ with $||x||^2 + ||y||^2 = 1$.
Conversely,  any nonzero $(x,y) \in \mathbb {F}^2$, after normalization,  produces a projection matrix.   This leads one to defining  equivalence classes  by letting $(x,y) \sim (x',y')$ if both pairs result in the same projection matrix, i.e. an element of $\mathbb{FP}^1$. Denote the equivalence class of $(x,y)$ by $[(x,y)]$. These equivalence classes generalize the ones defined in Remark \ref{r.f2rep} and $\mathbb{FP}^1$ can be  equivalently defined as the set of these equivalent classes.  
We are now ready to present an explicit isometry between $\mathbb{FP}^1$ and $\mathbb S^D$.
\begin{lemma}
For $x \in \mathbb F$, we write $x = (x_0,...,x_{D-1})$, with $\Re x = x_0$. Define $\tau: \mathbb{FP}^1 \to \mathbb S^D$ by
\begin{align*}
	\tau([(1,x)]) &= \frac{1}{1+||x||^2} \left(1-||x||^2,2x_0,2x_1,...,2x_{D-1}\right)\\
	\tau([(0,1)]) &= \left(-1,0,0,...,0\right)
\end{align*}
Then $\tau$ is an isometry from $\mathbb{FP}^1$ to $\mathbb{S}^D$ with their respective geodesic distances.
\end{lemma}
\begin{proof}
Since $\mathbb{FP}^1 = \{[(1,x)]| x\in \mathbb{F}\} \cup \{[(0,1)]\}$, see \cite{Baez}, it is easy to see that $\tau$ is a well-defined bijection. We can now prove that it is an isometry,
\begin{align*}
\cos(\vartheta([(1,x)],[(1,y)])) &= 2\left(\Re \Tr \frac{1}{(1+||x||^2)(1+||y||^2)}\begin{pmatrix} 1 & x\\ \overbar{x} & \overbar{x}x\end{pmatrix}\begin{pmatrix} 1 & y\\ \overbar{y} & \overbar{y}y\end{pmatrix}\right)-1\\
&= \frac{2}{(1+||x||^2)(1+||y||^2)}\left(\Re (1+x\overbar{y}+\overbar{x}y+\overbar{x}x\overbar{y}y)\right) - 1\\
&= 2\frac{1+||x||^2||y||^2+ 2\sum_{i=0}^{D-1} x_iy_i}{(1+||x||^2)(1+||y||^2)}-1\\
&= \frac{(1-||x||^2)(1-||y||^2) + 4\sum_{i=0}^{D-1} x_iy_i}{(1+||x||^2)(1+||y||^2)} = \cos(\eta(\tau([(1,x)]),\tau([1,y])))
\end{align*}
and
\begin{align*}
\cos(\vartheta([(1,x)],[(0,1)])) &=  2\left(\Re \Tr \frac{1}{(1+||x||^2)}\begin{pmatrix} 1 & x\\ \overbar{x} & \overbar{x}x\end{pmatrix}\begin{pmatrix} 0 & 0\\ 0 & 1\end{pmatrix}\right)-1\\
&= \frac{2||x||^2}{1+||x||^2}-1 = \frac{||x||^2-1}{1+||x||^2} = \cos(\eta (\tau([(1,x)]),\tau([0,1]))).
\end{align*}
Finally, $\cos(\vartheta([(0,1)],[(0,1)])) = 1 = \cos(\eta(\tau([(0,1)]),\tau([0,1])))$.
\end{proof}

\begin{remark} Observe that in the case of $\Omega = \mathbb{RP}^1$, the line corresponding to $(1,x)$ has direction  $\theta \in [-\pi/2,\pi/2)$ with $\cos \theta = \frac{1}{\sqrt{1+x^2}}$. Thus $\tau ([(1,x)]) = \left(\frac{1-x^2}{1+x^2}, \frac{2x}{1+x^2} \right) = (\cos 2\theta, \sin 2\theta)$. Hence, the isometry between $\mathbb{RP}^1$ and $\mathbb S^1$   essentially amounts to doubling the angle, which can be seen already from \eqref{e.thetabeta},  and the isometries are similar in nature in the other cases. 
\end{remark}


With this isometry at hand, the known results about the geodesic distance Riesz energy on the sphere can be translated to statements about the geodesic Riesz energy on $\mathbb{FP}^1$. We recall that Theorem 1.1 of \cite{BilD}  states that on the sphere $\mathbb S^D$ the minimizers of the geodesic Riesz energy are characterized as follows:
\begin{itemize}
\item $-1<s<D$: the uniform measure $\sigma$,
\item $s=-1$:   centrally symmetric measures, i.e. Borel probability measures $\mu $  on  $\mathbb S^d$ satisfying  $\mu (E) = \mu (-E)$ for all Borel sets $E\subset \mathbb S^D$,
\item $s<-1$: all measures of the form $\mu = \frac12 \big( \delta_p + \delta_{-p}\big)$ for $p \in \mathbb S^d$.
\end{itemize}
It remains to understand the pull-backs of these measures under the isometry. Since both on $\mathbb S^D$ and on $\mathbb{FP}^1$ the uniform measures $\sigma$ are the only measures invariant under  isometries of the space, the pull-back of a uniform measure is again uniform.


At the phase transition $s = -1$, the minimizers on the sphere are  centrally symmetric measures. 
The  corresponding condition on projective spaces is \emph{orthogonal symmetry}. A measure $\mu$ is orthogonally symmetric if for all Borel $E \subseteq \mathbb{FP}^1$, $\mu(E) = \mu(E^\perp)$, where
\[
E^\perp = \{x  \in \mathbb{FP}^1 | \vartheta(x,y) = \pi \text{ for some } y \in  E \},
\]
i.e. the set of points a diameter away from a point in $E$. 

We call these measures orthogonally symmetric because  in the case $\mathbb F = \mathbb R$, $\mathbb C$, or $\mathbb H$, the condition that two elements of $\mathbb{FP}^d$ are a diameter apart is equivalent to saying that their representers in $\mathbb F^{d+1}$ (equivalently, the  corresponding lines in $\mathbb F^{d+1}$) are orthogonal. At the same time, being a diameter apart on the sphere corresponds to opposite points. Hence, central symmetry on the sphere $\mathbb S^D$ is translated into orthogonal symmetry on $\mathbb{FP}^1$.  

In particular, when $\mathbb{F} \neq \mathbb O$,  two antipodal points on $\mathbb S^D$ correspond to points in $\mathbb{FP}^1$ represented by two orthogonal ``lines'' (one-dimensional subspaces) in $\mathbb{F}^2$, i.e. an orthonormal basis of $\mathbb{F}^2$.   It is obvious from the definition  that the measure $\mu = \frac{1}{2}( \delta_p  + \delta_q)$ with $\vartheta(p,q) = \pi$ (i.e., for $\mathbb{F} \neq \mathbb O$, the representers of $p$ and $q$ in $\mathbb F^2$ are orthogonal) is orthogonally symmetric. 


 
This discussion shows that  Theorem 1.1 of \cite{BilD} can be restated as follows.
\begin{theorem}\label{t.d1expand}
For $\Omega = \mathbb{FP}^1$, the minimizers of the geodesic distance energy integral $I_{F_s} (\mu)$ are characterized by:
\begin{enumerate}[(i)]
\item\label{i} $-1 < s < D$: The unique minimizer of $I_{F_s} (\mu)$ is $\mu = \sigma$. 
\item\label{ii} $s = -1$:   The energy $I_{F_s} (\mu)$ is maximized if and only if $\mu$ is orthogonally symmetric.
\item\label{iii}  $s > -  1$: The energy  $I_{F_s} (\mu)$ is maximized if and only if  $\mu = \frac{1}{2}( \delta_p  + \delta_q)$, where $\vartheta(p,q) = \pi$. 
\end{enumerate}
\end{theorem}

Parts \eqref{i} and \eqref{ii} of Theorem \ref{t.d1expand} give an alternative proof of and refine  the case $d=1$  of Theorem \ref{t.geodriesz} (the case of $s=-1$ follows from \eqref{ii} since the uniform measure on $\mathbb{FP}^1$ is easily seen to be orthogonally symmetric).\\

%

\par

The case  $\mathbb{F} = \mathbb{R}$ relates to the original Fejes T\'oth conjecture for $d=1$. Indeed, one just needs to  reformulate these results on $\mathbb S^1$ using representers of the points of $\mathbb{RP}^1$ in $\mathbb S^1$  and the original kernel $\theta (x,y)$, i.e. the acute angle between lines in $\mathbb R^2$, as defined in \eqref{e.theta}, which is up to constant just the geodesic distance on $\mathbb{RP}^1$ as shown in \eqref{e.rhotheta}. 

Under this interpretation and recalling that $s=-\lambda$, Theorem \ref{t.d1expand} for $\mathbb F = \mathbb R$  is exactly   Theorem  \ref{t.d1} (with uniqueness understood up to central symmetries). Since, as discussed above, the measures in part \eqref{iii} are orthogonally symmetric and the representers of $p$ and $q$ correspond to an orthonormal basis on $\mathbb R^2$, part \eqref{ii} of Theorem \ref{t.d1expand}  shows that $\mu_{ONB}$ maximizes \eqref{e.ift}, i.e. it gives an alternative proof  that the integral Fejes T\'oth conjecture (Conjecture \ref{c.IFT}) holds on $\mathbb S^1$ \cite{BilM}. \\




\par

The discrete version of both  the    geodesic distance Riesz energy on $\mathbb{S}^1$ \cite{Jia} and the Fejes T\'oth conjecture (Conjecture \ref{c.FT})  on $\mathbb S^1$  (i.e. geodesic  Riesz energy on $\mathbb{RP}^1$) \cite{BilM,FodVZ,Pet} have previously been studied and the optimizers have been completely characterized for all $N \in \mathbb N$ in both settings.  The isometry between $\mathbb{RP}^1$ and $\mathbb S^1$  provides a direct correspondence between these two families of point configurations, in other words, the two problems are really the same -- it appears that this connection has been overlooked in the literature.

Interestingly, the proof of the case $s = -1$ for the geodesic distance on $\mathbb S^d$  (in both integral and discrete setting) given in \cite{BilDM} uses a version  the so-called Stolarsky principle \cite{Sto}, which relates the discrepancy with respect to hemispheres to the geodesic distance Riesz energy. Similarly,  one of the proofs of the Fejes T\'oth conjecture on the circle $\mathbb S^1$ \cite{BilM} employs a version of the Stolarsky principle for the discrepancy with respect to unions of opposite quadrants (i.e. two opposite arcs of length $\pi/2$). Observe that under the isometry described above these shapes are exactly the preimages of hemispheres on $\mathbb S^1$ (semicircles). \\

Finally we remark that part \eqref{ii} of Theorem \ref{t.d1expand} implies that the  kernel $F_{-1} (\cos (\vartheta (x,y))) = -\vartheta (x,y)$ is  conditionally positive definite (albeit not strictly) on $\mathbb{FP}^1$, due to Theorem \ref{thm:uniform measure min, continuous}, since $\sigma$ is orthogonally symmetric and hence minimizes $I_{F_{-1}}$ (albeit not uniquely).  This property is particular to the case $d=1$.

\begin{lemma}\label{l.lambda<1}
If $\Omega = \mathbb{FP}^d$ with $d\ge 2$, then the uniform measure $\sigma$  is not a minimizer of $I_{F_{-1}} (\mu)$ on $\Omega$. Moreover, for any $\varepsilon >0$, it cannot be a minimizer of $I_{F_{s}}$ simultaneously for all $s \in (-1,-1+\varepsilon)$.
\end{lemma}

\begin{proof}
For the first statement, according to Theorem \ref{thm:uniform measure min, continuous}, one needs to check that the kernel $F_{-1} (t) = -\arccos t$ is not conditionally positive definite, or equivalently, that not all of its Jacobi coefficients  with $n\ge 1$ are non-negative. This has  been done in \cite[pp. 179--181]{Gan}. An alternative approach would consist of observing that it is enough to consider the case $d=2$. Indeed, $\mathbb{FP}^d$ with $d>2$ contains copies of $\mathbb{FP}^2$, and if a kernel is not positive definite on $\mathbb{FP}^2$, it can't be positive definite on any bigger domain. Then one could directly find negative coefficients in the Jacobi expansion of $F_{-1}$ on $\mathbb{FP}^2$ for all four instances of $\mathbb F$ (e.g., for $\mathbb{RP}^2$ this is already the second Jacobi coefficient). 

Hence, $\sigma$ is not a  minimizer for   $s = -1$. If $\sigma$ were a minimizer  for any sequence of values $s_j \rightarrow -1$, this would imply by continuity  that it is also a minimizer for $s =-1$, which is a contradiction. 
\end{proof}



\section{Chordal Distance Riesz Energy}\label{s.chordal}

In this section we  discuss the characterization of  the minimizers of the Riesz energy with respect to the chordal distance on projective spaces. Unlike the main focus of this paper, i.e. the geodesic distance Riesz energy, the minimizers of the chordal Riesz energy are more straightforwad --   they follow the exact same pattern as those for the Euclidean Riesz energy on the sphere, as alluded to in Section \ref{s.typical}. 

Recall that the {\emph{chordal distance}} on $\Omega$ is defined as 
\begin{equation}
\rho(x,y) = \sin \left( \frac{\vartheta(x,y)}{2} \right) =   \left( \frac{1-\cos(\vartheta(x,y))}{2} \right)^{1/2},
\end{equation}
where, as before, $\vartheta(x,y)$ is the geodesic distance on $\Omega$.  Observe that on the sphere this is a scaled version of the Euclidean metric, i.e. $\| x-y \| =  2 \rho (x,y) $ when $\Omega = \mathbb S^d$.

Just as in \eqref{e.fs1}, we  define \textit{chordal} Riesz $s$-kernels as 
\begin{equation}\label{eq:Chordal Riesz Def}
R_s(\cos(\vartheta(x,y))) = \frac{1}{s} \rho(x,y)^{-s} \,\,\, \textup{ for } s\neq 0 \,\,\qquad \textup{and} \,\,\qquad  
R_0(\cos(\vartheta(x,y))) =  - \log(\rho(x,y)). 
\end{equation}
where $\rho$ is the chordal metric on $\Omega$. 

 We are now ready to characterize    the minimizers  of the chordal distance Riesz energy (in the case of the Euclidean distance on the sphere this characterization is well known \cite{Bj,BorHS}). 

\begin{theorem}\label{thm:Chordal Riesz Minimizers}
Assume that $s < D = \dim( \Omega) = 2 \alpha +2$. Then the minimizers of the chordal Riesz energy $I_{R_s}$ on $\Omega$ are as follows:
\begin{itemize}
\item If $s>-2$, then the uniform measure $\sigma$ is the unique minimizer of $I_{R_s}$.
\item If $s = -2$, then the minimizing measures of $I_{R_s}$ are exactly those satisfying
\begin{equation}\label{e.bary}
\int_{\Omega} \int_{\Omega} P_1^{\alpha, \beta}\Big(\cos(\vartheta(x,y)) \Big) d \mu(x) d \mu(y) = 0.
\end{equation}
\item If $s < -2$ and $\Omega = \mathbb{FP}^{d}$, then the minimizing measures of $I_{R_s}$ are exactly those of the form $\mu_{ONB} = \frac{1}{d+1} \sum_{j=1}^{d+1} \delta_{e_j}$, where $\{e_1, ..., e_{d+1}\} \subset \mathbb{FP}^{d}$ is a set where all element are diameter apart, i.e. points corresponding to an orthonormal basis in $\mathbb{F}^{d+1}$.
\item If $s < -2$ and $\Omega = \mathbb{S}^{d}$, then the minimizing measures of $I_{R_s}$ are exactly those of the form $\frac{1}{2} \Big( \delta_{p} + \delta_{-p}\Big)$ for any $p \in \mathbb{S}^d$.
\end{itemize}
\end{theorem}

{Observe that the  last two statements can be combined into one: indeed, in both cases, the minimizers are exactly uniform measures on any maximal  set whose points are separated by the diameter of  $\Omega$.} Therefore, there is indeed no difference between the behavior of optimizers of the Euclidean Riesz energy on the sphere and the chordal Riesz energy on projective spaces. 

We also remark on the condition \eqref{e.bary}. It is obviously equivalent to 
\begin{equation}\label{e.bary1}
 \int_{\Omega} \int_{\Omega} \cos(\vartheta(x,y)) \, d \mu(x) d \mu(y) =  \int_{\Omega} \int_{\Omega} \cos(\vartheta(x,y)) \, d \sigma(x) d \sigma (y).
 \end{equation}
  In the case, of the sphere this amounts to $$  \int_{\mathbb S^d}  \int_{\mathbb S^d} \langle x,y \rangle  d \mu(x) d \mu(y) = \left\| \int_{\mathbb S^d} x  d \mu(x) \right\|^2 = 0 ,$$  which implies that $\int_{\mathbb S^d} x  d \mu(x) = 0$, i.e. the minimizers are exactly the barycentric measures.

\begin{proof}
The results for the sphere are classical for $- 2 < s < D$ (see e.g. \cite[Proposition 4.6.4]{BorHS}), and were proven for $s < -2$ in \cite[Theorem 6, and the Corollary to Theorem 7]{Bj}.  The case of $\Omega = \mathbb{FP}^{d}$ with  $0 \leq s < \dim(\Omega)$ was treated  in  \cite[Theorem 2.11]{AndDGMS}, so we need only address the case $s < 0$ on projective spaces (though the method will also cover $s < 0$ for spheres).


Assume first  that $-2 < s < 0$. Then
\begin{equation}\label{e.rtaylor}
R_s(t) =  \frac{1}{s} \Big( \frac{1- t}{2} \Big)^{-\frac{s}{2}} =  \frac{1}{s} \Big( 1 - \frac{t+1}{2} \Big)^{-\frac{s}{2}} =\sum_{k=0}^{\infty}    \frac{1}{s}  \binom{ -\frac{s}{2}}{k}  \Big( \frac{-1}{2} \Big)^k (t+1)^k.
\end{equation}
It is easy to see that all the coefficients in the last Taylor series above are positive, and a  quick check with the ratio test shows us that this series converges uniformly on $[-1,1]$.

Observe that  $P_0^{(\alpha, \beta)} (t) =1$ and $P_1^{(\alpha, \beta)} (t)  = (\alpha +1) + (\alpha + \beta + 2) \frac{t-1}{2}$, we see that the function 
$$ 1 + t = \frac{\beta +2}{\alpha + \beta + 2}  P_0^{(\alpha, \beta)} (t)  + \frac{2}{\alpha + \beta + 2} P_1^{(\alpha, \beta)} (t) $$ is positive definite for all geometrically relevant values of $\alpha $ and $\beta$, since $\alpha$, $\beta \ge -\frac12$, see \eqref{eq:JacobiAlphaBeta}. Therefore, $(1+ t)^k $ is also positive definite for all $k \in \mathbb N$  due to Schur's product theorem, and thus  $R_s(t) $ is conditionally positive definite due to \eqref{e.rtaylor} as a positive linear combination (up to the constant term) of positive definite functions. This implies that for $n\ge 1$
\begin{align}
 \widehat{R_s}(n) 
& =  \frac{1}{s} \frac{m_n}{P_n^{(\alpha, \beta)}(1)^2}   \sum_{k=0}^{\infty} \binom{ \frac{-s}{2}}{k}  \Big( \frac{-1}{2} \Big)^k \int_{-1}^{1} (t+1)^k P_n^{(\alpha, \beta)}(t) d\nu^{\alpha, \beta}(t) \ge 0.  \label{eq.fubi} 
\end{align}
Moreover, notice that, due to positive definiteness of $(1+t)^k$, all the integrals in \eqref{eq.fubi} are non-negative, and in addition the term with $k=n$ must be strictly positive (if it were zero, $P_n^{(\alpha, \beta)}$ would be orthogonal to all polynomials of degree $n$). Therefore, $\widehat{R_s}(n) > 0 $ for all $n\ge 1$,  and  $\sigma$ is the unique minimizer of $I_{R_s}$ in view of  Theorem \ref{thm:uniform measure min, continuous}. \\

We now consider the case $s=-2$. Since
$$2 R_{-2}(\cos(\vartheta(x,y)) = \big( \rho (x,y) \big)^2  = \frac{\cos(\vartheta(x,y))-1}{2} = \frac{1}{\alpha+\beta+2} P_1^{(\alpha, \beta)} \Big( \cos(\vartheta(x,y))\Big)  -  \frac{\alpha+ 1}{\alpha+\beta+2},$$
  by Theorem \ref{thm:uniform measure min, continuous}, $\sigma$ minimizes $I_{R_{-s}}$. Thus, if $\mu$ is a minimizer of $I_{R_{-2}}$, we must have
$$\int_{\Omega} \int_{\Omega}  P_1^{(\alpha, \beta)} (\cos(\vartheta(x,y))) d\mu(x) d \mu(y) = \int_{\Omega} \int_{\Omega} P_1^{(\alpha, \beta)}(\cos(\vartheta(x,y))) d\sigma(x) d \sigma(y) = 0 
,$$
which yields the  second claim \eqref{e.bary}.\\

For the final claim, we note that  
 for $s < -2$,
$$ s R_s(\cos(\vartheta(x,y)) = 
 \big( \rho (x,y) \big)^{-s}   \leq  \big( \rho (x,y) \big)^2  = (-2 ) R_{-2}(\cos(\vartheta(x,y))$$
with equality if and only if $\vartheta(x,y) \in \{0, \pi\}$. Hence  for all $\mu \in \mathbb{P}(\Omega)$,
$$    I_{R_{s}}(\mu) \geq - \frac{2}{s}  I_{R_{-2}}(\mu) \geq  - \frac{2}{s}  \inf_{\nu \in \mathcal{P}(\Omega)} I_{R_{-2}}(\nu).$$
The first inequality becomes an equality if  $\vartheta(x,y) \in \{0, \pi\}$ for all $x,y \in \operatorname{supp}(\mu)$. This means that $\operatorname{supp}(\mu)$ is discrete, with at most $d+1$ elements (or $2$ if $\Omega$ is a sphere). 
The second inequality becomes an equality if $\mu$ minimizes $I_{R_{-2}}$, in particular, it satisfies \eqref{e.bary}.
One can easily check that the measures in  the last two bullet points of Theorem \ref{thm:Chordal Riesz Minimizers} satisfy both of these conditions.
For example, for $\mu = \mu_{ONB}$ the left-hand side of \eqref{e.bary} becomes
$$ \frac{d}{d+1} P_1^{\alpha, \beta} (-1) +  \frac{1}{d+1} P_1^{\alpha, \beta} (1)  =  (\alpha + 1) -  \frac{d}{d+1} (\alpha + \beta +2)  = 0,$$
 where we use the fact that $\beta = \frac{\alpha +1}{d} -1$ for $\mathbb{FP}^d$, see \eqref{eq:JacobiAlphaBeta}.

Moreover, a simple convexity argument shows that  if the support of $\mu$ is not maximal (i.e. contains less than $d+1$ or, respectively, $2$ points) or if the weights are not equal, then the chordal Riesz energy can be decreased by redistributing mass, i.e. such measures cannot be minimizers. Indeed, if $\mu = \displaystyle{\sum_{i=1}^{d+1} \gamma_i \delta_{z_i}}$, where $\{z_1,\dots, z_{d+1}\}$ are diametrically  separated points, and if $\gamma_1 \neq\gamma_2$ (one of these could be zero), it is easy to check that replacing both weights by $\frac12 (\gamma_1 +\gamma_2)$ decreases the energy, since $\gamma_1 \gamma_2 < \left( \frac{\gamma_1 + \gamma_2}{2} \right)^2$. 
This finishes the proof of the  claim (a similar argument is used in the proof  of Lemma \ref{lem:Linear programming for ONB}). \\

 An alternative (perhaps more illuminating, but requiring deeper background) way to finish the proof would be the following: since $\mu$ has discrete support, condition \eqref{e.bary} states that $\mu$  corresponds to a weighted $1$-design on $\Omega$. 
 But since its cardinality is at most $d+1$ (or $2$ if $\Omega$ is a sphere) which happens to be exactly the cardinality of  the so-called tight $1$-designs. However, any weighted design has cardinality at least as large as that of a tight design, with equality if and only if the weighted design is actually a tight design (in particular, it has equal weights).  Since the configurations described in the last two bullet points of Theorem \ref{thm:Chordal Riesz Minimizers}
are exactly tight $1$-designs on the corresponding space $\Omega$, this finishes the proof. See \cite{Lev,Lyu} for the relevant definitions
and statements. 

\end{proof}

\section{Discrete Minimizers of the Geodesic Riesz Energy}\label{sec.discrete}

We now return to the geodesic Riesz energy on projective spaces.  In contrast to the results of the previous sections, which focused on the range of Riesz expontents for which  the uniform measure $\sigma$ minimizes such energies, we shall now turn our attention to the other end of the spectrum -- the case when the discrete measures $\mu_{ONB}$ optimizes the energy, i.e. the situation related to  the integral version of the  Fejes T\'oth conjecture \eqref{e.ift}. Many of the ideas presented in this section have appeared in \cite{LimM22}, and some even earlier in \cite{BilM}, in the context of the integral Fejes T\'oth conjecture \eqref{e.ift} on the sphere. Here we further develop and simplify them, translate them to the language of geodesic Riesz energy on $\mathbb{RP}^d$ and extend them to other projective spaces. 

We start with a  lemma that shows that if the conjecture holds (i.e. $\mu_{ONB}$ is an optimizer)  for some exponent, then measures of the type  $\mu_{ONB}$ {\emph{uniquely}}  optimize the energy for all larger powers (smaller Riesz exponents). This statement along with its proof is similar in spirit to the case $s<-2$ (third bullet point) of Theorem \ref{thm:Chordal Riesz Minimizers}.
\begin{lemma}\label{lem:Linear programming for ONB}
Let $\Omega = \mathbb{FP}^d$ and  $\mu_{ONB} = \frac{1}{d+1} \sum_{j=1}^{d+1} \delta_{e_j}$, where $\{e_1, ..., e_d\} \subset \mathbb{FP}^{d}$ is a set where all element are diameter apart, i.e. points corresponding to an orthonormal basis in $\mathbb{F}^{d+1}$. If $\mu_{ONB}$ is a  minimizer of $I_{F_{s^*}}$ for some $s^* < 0$, then such measures are exactly the minimizers of $I_{F_s}$ for all $s < s^*$.
\end{lemma}

\begin{proof}
For $s < s^*<0$, we have that
$$ s^* \pi^{s^*} F_{s^*}(t) = \left(\frac{\arccos(t)}{ \pi} \right)^{-s*} \geq \left(\frac{\arccos(t)}{ \pi} \right)^{-s} =   s \pi^{s} F_{s}(t), $$
with equality if and only if $t \in \{-1, 1\}$, i.e. $\arccos(t) = \vartheta(x,y) \in \{ 0, \pi\}$, meaning that $x = y$ or $x$ and $y$ are diameter apart. Thus for all probability measures on $\mathbb{FP}^d$,
\begin{equation*}
I_{F_{s}}(\mu) \geq  \frac{{s}^*}{s} {\pi^{s^*-s}} \, I_{F_{s^*}}(\mu) \geq     \frac{{s}^*}{s} {\pi^{s^*-s}} \,\, I_{F_{s^*}}(\mu_{ONB}).   
\end{equation*}
The first inequality becomes an equality if and only if for all $x,y \in \operatorname{supp}(\mu)$, $\vartheta(x,y) \in \{ 0, \pi\}$. Let us assume $\mu$ is such a measure. In this case, $N = \# \operatorname{supp}(\mu) \leq d+1$, and if $\mu = \sum_{j=1}^{N} \lambda_j \delta_{x_j}$, then by Jensen's inequality, using the fact that $\displaystyle{\sum_{i\neq j} \lambda_i  = 1 - \lambda_j}$, we find that 
$$ I_{F_{s^*}}(\mu) = \frac{\pi^{-s^*}}{s^*}\sum_{j=1}^{N} (1- \lambda_j) \lambda_j \geq \frac{\pi^{-s^*}}{s^*} \left(1- \frac{1}{N} \right) \geq \frac{\pi^{-s^*}}{s^*} \left(1- \frac{1}{d+1}\right) = I_{F_{s^*}}(\mu_{ONB}).$$
The first inequality becomes an equality if and only if the weights $\lambda_1 = ... = \lambda_N = \frac{1}{N}$, since $\phi(x) = x(1-x)$ is strictly concave, and the second inequality becomes an equality only when $N = d+1$, proving the lemma.
\end{proof}


Next, we observe that if $\mu_{ONB}$ minimizes the energy $I_{F_s} $ with $s=-1$ (i.e. if the integral version of the Fejes T\'oth conjecture \eqref{e.ift} holds  in case of $\Omega = \mathbb{RP}^d$), then the energy $I_{F_{-1}} $ necessarily admits various other minimizers.

\begin{lemma}\label{lem:Behavior for geodesic distance for ONB}
Let $\Omega = \mathbb{FP}^d$ with $d \geq 2$. If $\mu_{ONB}$ is a minimizer of $I_{F_{-1}}$, then  $I_{F_{-1}}$ also has minimizers whose support is not discrete (in particular, there exist minimizers whose support contains  one-dimensional subsets). 
\end{lemma}

\begin{proof}
Let $\omega_1$ be the uniform measure on $\mathbb{FP}^1$ and  set  $\omega_2 = \frac{1}{2}( \delta_{z_1} + \delta_{z_2})$ with $z_1, z_2 \in \mathbb{FP}^{1}$ such that $\vartheta(z_1, z_2) = \pi$. Part \eqref{ii} of Theorem \ref{t.d1expand} shows that $I_{F_{-1}} (\omega_1) = I_{F_{-1}} (\omega_2)$, since both of these measures are orthogonally symmetric and thus both minimize $I_{F_{-1}}$.  

Using this fact, the strategy for constructing alternative minimizers is clear: assuming that the measure $\displaystyle{\mu_{ONB} = \frac{1}{d+1} \sum_{i=1}^{d+1} \delta_{x_i} }$   (where $\vartheta(x_i, x_j) = \pi$ for  all $i \neq j$) is a minimizer of $I_{F_{-1}}$,  one can replace any two  discrete masses at points $x_i$ and $x_j$ by a uniform measure on a copy of $\mathbb{FP}^1$ which contains these points, and this would not change the value of the energy, thus yielding another minimizer. \\

Indeed,  using the notation above, set $\nu = \frac{1}{d-1}\sum_{j=3}^{d+1} \delta_{x_j}$. Let $ \widetilde{\omega}_1$ be the uniform measure on the copy of $\mathbb{FP}^1$ containing $x_1$ and $x_2$ and set $ \widetilde{\omega}_2 = \frac12 (\delta_{x_1} + \delta_{x_2} )$ (i.e., these measures are copies of $\omega_1$ and $\omega_2$). With these definitions $$\mu_{ONB} = \frac{2}{d+1}  \widetilde{\omega}_2 + \frac{d-1}{d+1} \nu, $$
and $\vartheta(x,y) = \pi$ for any $y \in \operatorname{supp}({\nu})$ and $x$ in $\operatorname{supp}(\widetilde{\omega}_1)$ (or $\operatorname{supp}(\widetilde{\omega}_2)$). Finally, set $$\mu_{1} = \frac{2}{d+1}  \widetilde{\omega}_1 + \frac{d-1}{d+1} \nu. $$Then we have
\begin{align*}
I_{F_{-1}}(\mu_{ONB}) &=
 \frac{4}{(d+1)^2}I_{F_{-1}}(\omega_2) + \frac{(d-1)^2}{(d+1)^2}I_{F_{-1}}(\nu) +  \frac{4(d-1)}{d+1} F_{-1}(-1) \\
& = \frac{4}{(d+1)^2}I_{F_{-1}}(\omega_1) + \frac{(d-1)^2}{(d+1)^2}I_{F_{-1}}(\nu) +  \frac{4(d-1)}{d+1} F_{-1}(-1)  = I_{F_{-1}}(\mu_1),
\end{align*}
hence $\mu_1$ is also a minimizer.
\end{proof}

We first observe that in the case of $\Omega = \mathbb{RP}^2$ this construction essentially yields the pole-equator measure on $\mathbb S^2$ introduced in Section \ref{s.<1}, see \eqref{eq.pole-equator}. 
We  also remark that the procedure described in the proof of Lemma \ref{lem:Behavior for geodesic distance for ONB} can be repeated and iterated.  In particular, if the conjecture holds on   $\Omega = \mathbb{RP}^3$, one could construct an alternative minimizer which  consists of two copies of the uniform measure on $\mathbb{FP}^1$. In other words, switching to the representers on the sphere, this would consist of uniform measures on two orthogonal copies of $\mathbb S^1$ in $\mathbb S^3 \subset \mathbb R^4$.  This would produce a purely one dimensional minimizer without any point masses. \\

Next we show that the measures of the type $\mu_{ONB}$ are unique energy  minimizers for $s=-2$. 

\begin{lemma}\label{lem:Behavior for geo dist squared for ONB}
If $\Omega = \mathbb{FP}^d$, then the minimizing measures of $I_{F_{-2}}$ are exactly those of the form $\mu_{ONB} = \frac{1}{d+1} \sum_{j=1}^{d+1} \delta_{e_j}$, where $\{e_1, ..., e_d\} \subset \mathbb{FP}^{d}$ is a set where all element are diameter apart, i.e. points corresponding to an orthonormal basis in $\mathbb{F}^{d+1}$.
\end{lemma}

\begin{proof}
Let
$$h(t) := \Big( \frac{1-t}{2} \Big)^{\frac{1}{2}} -  \frac{\arccos t}{\pi} .$$
We find that
$$ h''(t) =  \frac{- \sqrt{2} \pi (1+t)^{3/2} + 8t}{8(1-t^2)^{3/2}}  $$
is negative on $[-1,1)$. This can be seen immediately for $t \leq 0$, and for $t > 0$, by considering
$$g(t):= - \sqrt{2} \pi (1+t)^{3/2} + 8t.$$
We see that $g'(t) = 8- \frac{3 \pi \sqrt{1+t}}{\sqrt{2}}$ and $g''(t) = - \frac{3 \pi}{2 \sqrt{2(1+t)}}$, so $g$ is concave with a maximum at $a = \frac{128}{9 \pi^2}-1 \approx 0.44$, where $g(a) \approx -4$. Thus, $g$ is negative on $(0,1)$ and so is $h''$, meaning that $h(t)$ is strictly concave on $(-1,1)$ and $h(-1)= h(1) = 0$, so $h$ is nonnegative on $[-1,1]$ and strictly negative on $(-1,1)$.

Thus, recalling that $R_s$ denotes the chordal Riesz kernel  defined in \eqref{eq:Chordal Riesz Def}, we find  that
\begin{equation}\label{e.linprog} 2 R_{-2}(t) = \frac{t-1}{2} \leq - \frac{(\arccos t)^{2}}{ \pi^{2}} =  2 \pi^{-2} F_{-2}(t) ,
\end{equation}
with equality if and only if $t \in \{-1, 1\}$, i.e. $\arccos(t) = \vartheta(x,y) \in \{ 0, \pi\}$, meaning that $x = y$ or $x$ and $y$ are diameter apart. Inequality \eqref{e.linprog} is illustrated in Figure \ref{f.linprog}.

\begin{figure}[h]
\includegraphics[width=.7\textwidth]{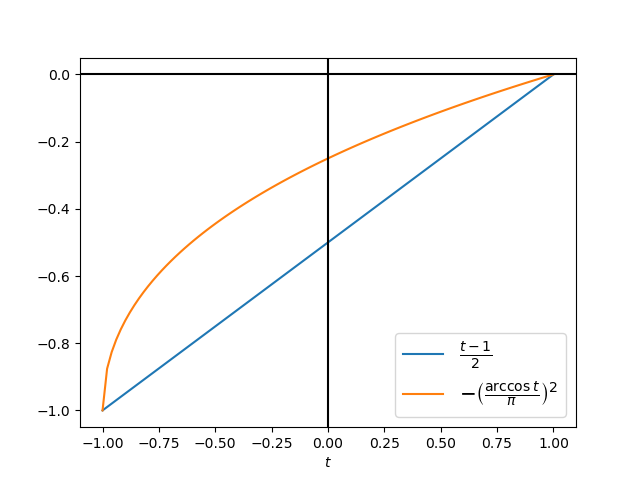}
\caption{Illustration of inequality \eqref{e.linprog}.}
\label{f.linprog}
\end{figure}

Thus, for all probability measures $\mu$ on $ \mathbb{FP}^{d}$, according to Theorem \ref{thm:Chordal Riesz Minimizers}, we have
\begin{equation}
\frac{1}{\pi^{2}} I_{F_{-2}}(\mu) \geq I_{R_{-2}}(\mu) \geq \min I_{R_{-2}} (\mu) =  I_{R_{-2}}(\sigma).
\end{equation}
The first inequality becomes an equality exactly when, for all $x,y \in \operatorname{supp}(\mu)$, $\vartheta(x,y) \in \{0, \pi\}$, which implies that the support of $\mu$ has at most $d+1$ points. The second one becomes an equality if $\mu$ minimizes $I_{R_{-2}}$,  which in view of the case $s=-2$ of Theorem \ref{thm:Chordal Riesz Minimizers} means that $\mu$ satisfies \eqref{e.bary}. One can then proceed exactly as in the end of the proof of Theorem  \ref{thm:Chordal Riesz Minimizers}  to show that the measures of the form $\mu_{ONB}$, 
and only they, minimize $I_{F_{-2}}$.
\end{proof}

The method  presented here, which relies on bounding the kernel in question from below by a positive definite kernel,  is often referred to as the {\emph{linear programming method}} \cite{DGS,CohK,CohKM}. The argument of Lemma \ref{lem:Behavior for geo dist squared for ONB} worked specifically for the kernel $F_{{-2}}$ since $\lambda= 2$ is the smallest power  such that the function $(\arccos t)^\lambda$ has a finite derivative at $t=1$,  which allows one to obtain inequality \eqref{e.linprog}. 

We now show that the analogue of the Fejes T\'oth conjecture holds  for a range of powers in all projective spaces. 

\begin{theorem}\label{t.FTlarge} Let $\Omega = \mathbb{FP}^d$.
There exists a unique $s^*  \in [-2,-1]$ such that for all $s < s^*$, $I_{F_s}$ is minimized exactly by the measures of the form $\mu_{ONB}$, and for $0 > s > s^*$, $I_{F_s}$ is not minimized by $\mu_{ONB}$.
\end{theorem}

\begin{proof}

The existence and uniqueness of such an $s^*$ and the fact that $s^* \geq -2$ follows from Lemmata \ref{lem:Linear programming for ONB} and \ref{lem:Behavior for geo dist squared for ONB}. We see that $s^*$ cannot be greater than $-1$, as that would imply that $I_{F_{-1}}$ is minimized only by measures of the form $\mu_{ONB}$, due Lemma \ref{lem:Linear programming for ONB}, which would contradict Lemma \ref{lem:Behavior for geodesic distance for ONB}.
\end{proof}

We believe that the condition $s^* \in [-2,-1]$ in Theorem \ref{t.FTlarge} should be replaced by $s^* \in (2,-1]$. A heuristic explanation lies in the fact that at the phase transition energies usually have large families of  minimizers  with very different structures, while for $s=-2$, according to Lemma \ref{lem:Linear programming for ONB}, the energy is minimized only by $\mu_{ONB}$. Hence, it would be reasonable to guess that the phase transition cannot happen at $s=-2$. In fact, this was proved  in \cite[Lemma 1.1 and Corollary 3.8]{LimM22} for $\Omega = \mathbb{RP}^d$ using precisely this heuristic, and it appears that the proof generalizes to other projective spaces (however, we choose not to pursue the details in this paper). 

Due to Lemmata  \ref{lem:Linear programming for ONB} and \ref{lem:Behavior for geo dist squared for ONB}, the integral  Fejes T\'oth conjecture (Conjecture \ref{c.IFT})  is equivalent to the fact that $s^* = -1$ for $\Omega = \mathbb{RP}^d$. We conjecture that, in fact, $s^*=-1$ in all  projective spaces.


\section{Results of numerical experiments on $\mathbb S^2$: $0<\lambda<1$}\label{s.<1}
In this section we discuss  results of numerical experiments, which shed some light on the behavior of minimizers below the value $\lambda =1$ (or equivalently above $s=-1$) which corresponds to the Fejes T\'oth conjecture. To be more precise, these results indicate that this behavior is rather mysterious.

Since the geometric setting of the sphere $\mathbb S^d$ is much more intuitive and more conducive to setting up the numerics, in this section we shall switch back to the original notation on the sphere, i.e. we shall  use the kernel $\theta (x,y)$   on $\mathbb S^d$ as defined in \eqref{e.theta} as well as the energy integral $I_\lambda$, see \eqref{e.def},  rather than  $\vartheta (x,y)$ and $I_{F_s}$ on $\mathbb{RP}^d$ as in most other parts of the paper (recall that $s = - \lambda$). \\

Most of our numerical results will be restricted to $\mathbb S^2$. Theorem \ref{t.main} shows that   $\sigma$ is a unique minimizer of the energy $I _\lambda (\mu)$ for $-2< \lambda \leq 0$. At the same time, the continuous version of the  Fejes T\'oth conjecture  would imply that $\mu_{ONB}$ is  a maximizer of $I_\lambda$ for all $\lambda \geq 1$ (unique for $\lambda >1$), see \S\ref{sec.discrete}. However, it turns out that neither of these measures  provides an optimizer on $\mathbb S^2$  for $0.6<\lambda  < 1$. 

Numerical computations show that the Gegenbauer coefficients of the kernel $f_\lambda (t) = ( \arccos |t| )^\lambda$ are not all nonpositive when $0.6 < \lambda \leq 1$, which excludes the possibility of the uniform measure being the optimizer of $I_\lambda$ on $\mathbb S^2$ in this range according to Theorem \ref{thm:uniform measure min, continuous}. (We would like thank Peter Grabner for some of these computations.) At least for values of $\lambda$ close to $1$, this statement is  made rigorous in Lemma \ref{l.lambda<1}.


This suggests, that unlike most of the examples given in \S\ref{s.typical}, the distance integrals  \eqref{e.di}  (i.e. the geodesic distance Riesz energies on the real projective space) undergo more than one phase transition as $\lambda$ changes.
The exact structure of  optimizers  in the range $0.6< \lambda < 1$   remains elusive even on the level of conjectures. 
We used numerical experiments to investigate potential optimizers on $\mathbb S^2$ in this range. 

We start by introducing some potential candidates for maximizing measures. We shall say that a Borel probability measure $\mu$  on $\mathbb S^2$ is  a {\emph{``pole-equator'' measure}} if for some $w\in [0,1]$ and $p \in \mathbb S^2$ it can be represented in the form
\begin{equation}\label{eq.pole-equator}
\mu = w\delta_p + (1-w)\nu_{p^{\perp}}, 
\end{equation} where  $\nu$ is a Borel probability measure  on $\mathbb S^1$ and $\nu_{p^{\perp}}$ is its copy  supported on $p^\perp = \{x \in \mathbb S^2 | \, x\cdot p = 0\}$. 
A simple calculation shows that
\begin{proposition}\label{p.poleequator}
If a probability measure $\nu$ maximizes $I_1$ on $\mathbb S^1$, then $I_1 \big(\frac{1}{3}\delta_p + \frac{2}{3}\nu_{p^{\perp}} \big)=I_1(\mu_{ONB})$ on $\mathbb S^2$.
\end{proposition}

\begin{proof}
Choose $p'$, $p'' \in p^\perp$ so as to complete $p$ to an orthonormal basis of $\mathbb R^3$.  Since the Fejes T\'oth conjecture holds on $\mathbb S^1$,  $I_1 \big(\frac12 \delta_{p'} + \frac12 \delta_{p''} \big) = I_1 (\nu)$
for any $\nu$ which maximizes $I_1$ on $\mathbb S^1$. It is then easy to see that on $\mathbb S^2$
\begin{align}
\nonumber I_1 \left(\frac{1}{3}\delta_p + \frac{2}{3}\nu_{p^{\perp}} \right) &= \frac{2\pi}{9} + \frac49 I_1 (\nu) = \frac{2\pi}{9} + \frac49 I_1 \left( \frac12 \delta_{p'} + \frac12 \delta_{p''} \right)\\ & =  I_1 \left(\frac{1}{3}\delta_p + \frac{2}{3} \left( \frac12 \delta_{p'} + \frac12 \delta_{p''} \right) \right)=I_1(\mu_{ONB}). \label{eq.backtoONB}
\end{align}
\end{proof}

Therefore, if the integral   Fejes T\'oth conjecture is true on $\mathbb S^2$ and  $I_1$ is indeed maximized by $\mu_{ONB}$, it would imply that a variety of pole-equator measures also  maximize $I_1$. Indeed, as discussed in Section \ref{s.d1},  possible choices of $\nu$ include $\sigma$, $\mu_{ONB}$, and, more generally, any orthogonally symmetric measure on $\mathbb S^1$. In particular, a pole-equator measure with $\nu = \sigma$, would yield a maximizer on $\mathbb S^2$, whose support has one-dimensional components, in contrast to the zero-dimensional $\mu_{ONB}$, see also Lemma \ref{lem:Behavior for geodesic distance for ONB}. We would like to remark that it is a typical situation for a phase transition to yield a large pool of optimizers of very different structure, which gives further evidence supporting the Fejes T\'oth conjecture.

Thus  it is natural to investigate pole-equator measures as possible  maximizers for $\lambda < 1$.  It is easy to determine which measure is optimal among all pole-equator measures. 

\begin{proposition}
Let $\sigma$ be the uniform surface measure on $\mathbb S^1$. Then for any $0 < \lambda < 1$, the measure $\frac{\lambda}{2\lambda+1}\delta_p + \frac{\lambda+1}{2\lambda+1} \sigma_{p^\perp}$ maximizes $I_\lambda$ among pole-equator measures.
\end{proposition}
\begin{proof}
Let $\mu = w\delta_p+(1-w)\nu_{p^\perp}$ be an arbitrary pole-equator measure. Then we have
\begin{align*}
I_\lambda(\mu) = 0\cdot w^2+2w(1-w)\left(\frac{\pi}{2}\right)^\lambda+ (1-w)^2I_\lambda (\nu)
\end{align*}
By Theorem \ref{t.d1},
 $I_\lambda (\nu)$ is maximized when $\nu = \sigma$ on $\mathbb S^1$ . 
Since $I_\lambda (\sigma) = \frac{1}{\lambda+1}\left(\frac{\pi}{2}\right)^{\lambda}$ on $\mathbb S^1$, optimizing  the quadratic expression shows that $w = \frac{\lambda}{2\lambda+1}$ maximizes $I_\lambda$.
\end{proof}

Moreover, we can use the same ideas to show that $\mu_{ONB}$ cannot be a maximizer  of $I_\lambda$ for $0<\lambda <1$ in all dimensions. 

\begin{proposition}\label{p.notONB}
Let $0< \lambda <1$. Then there exists a measure $\mu^* \in \mathcal P (\mathbb S^d)$ such that $I_\lambda (\mu^*) > I_\lambda (\mu_{ONB})$. 
\end{proposition}

\begin{proof}
We construct the measure $\mu^*$ explicitly by replacing two elements of an orthonormal basis by a uniform measure on a copy of $\mathbb S^1$ containing those two elements, i.e. if $e_1, \dots , e_{d+1}$ is an orthonormal basis of $\mathbb R^{d+1}$, we set
\begin{equation*}
\mu^* = \frac{1}{d+1} \sum_{i=1}^{d-1} \delta_{e_i} + \frac{2}{d+1} \sigma_{\{e_1, \dots, e_{d-1}  \}^\perp}.
\end{equation*} 
Then a computation almost identical to \eqref{eq.backtoONB} finishes the proof since $\sigma$ uniquely (up to symmetries) maximizes $I_\lambda$ on $\mathbb S^1$ for  $0< \lambda <1$, which follows from Theorem \ref{t.d1} (or  part \eqref{i} of Theorem \ref{t.d1expand} with $\Omega = \mathbb{RP}^1$).  Observe that similar constructions were used in the proof of Lemma \ref{lem:Behavior for geodesic distance for ONB}.
\end{proof}

Since  we know from Theorem \ref{t.main} that the uniform measure $\sigma$ on $\mathbb S^2$  optimizes  $I_\lambda$ at least for $\lambda \leq 0$,  even though the numerics discussed above suggests that it is not going to be an optimizer for $\lambda \in (0.6,1)$, it is natural to  numerically compare the values of the energy   $I_\lambda$   with  $\lambda \in (0,1)$ for  $\sigma$ and the optimal pole equator measure.  One can  see that the  change   occurs around $\lambda_0 = .76$, i.e. $I_{\lambda} (\sigma)$ becomes larger for $0< \lambda < \lambda_0$, while  pole-equator measures yield larger values for  $\lambda_0 < \lambda <1$, see Figure \ref{fig:IalphaPlot}.

 We ran numerical experiments to attempt to understand the structure of  maximizing measures  for $\alpha$ near $.76$. For our numerical experiments, we used the cyipopt python package. We replaced the energy integral \eqref{e.def}  with a  sum over $200$ points (discrete energy)  optimizing their location on the sphere as to maximize  the energy, and we ran $100$ trials. Additionally, we normalized the kernel $\theta$ as $\Lambda(x,y) = \left(\frac{2}{\pi}\arccos(|x\cdot y|)\right)^\lambda$ so that the values are between 0 and 1. The results are shown in Figure \ref{fig:experimentvisuals}.
 
\begin{figure}[h]
\begin{subfigure}[b]{.49\linewidth}
\centering
\includegraphics[width=.6\textwidth]{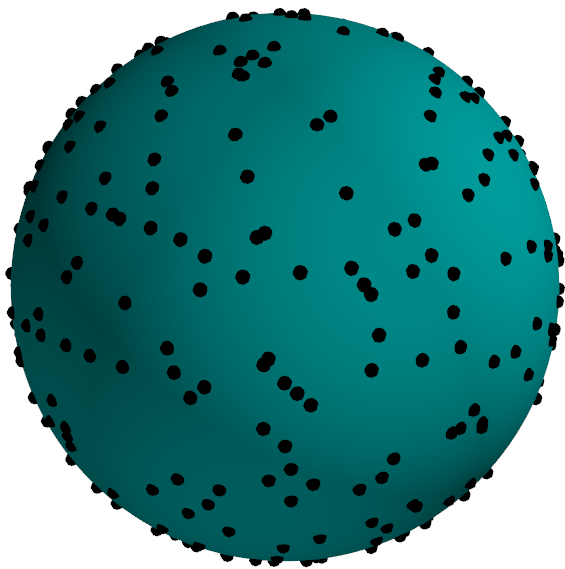}
\caption{$\mu_{.70}$}
\end{subfigure}\hfill
\begin{subfigure}[b]{.49\linewidth}
\centering
\includegraphics[width=.6\textwidth]{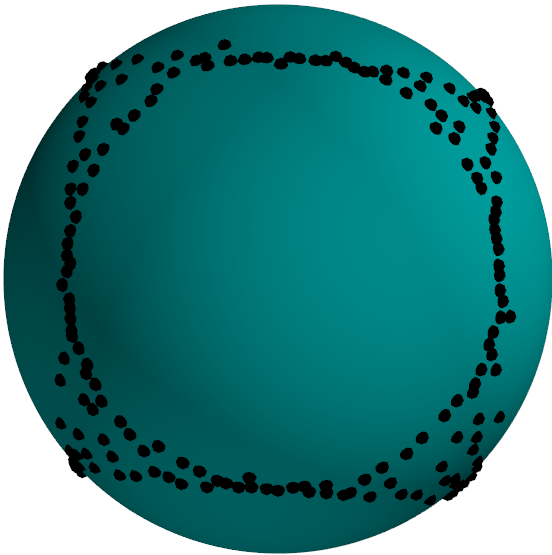}
\caption{$\mu_{.74}$}
\end{subfigure}
\begin{subfigure}[b]{.49\linewidth}
\centering
\includegraphics[width=.6\textwidth]{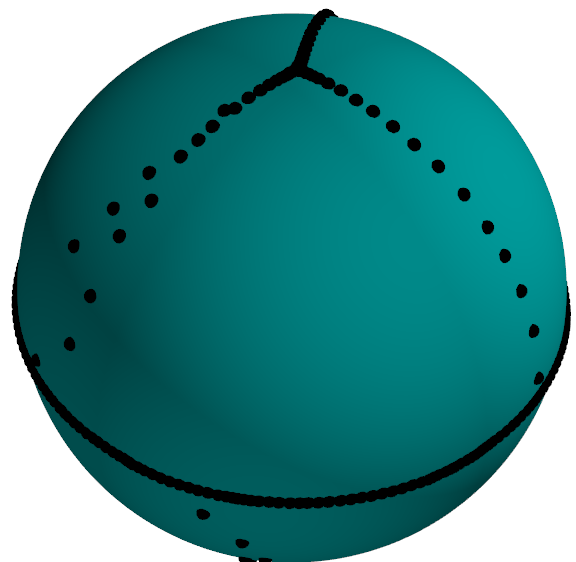}
\caption{$\mu_{.78}$}
\end{subfigure}
\begin{subfigure}[b]{.49\linewidth}
\centering
\includegraphics[width=.6\textwidth]{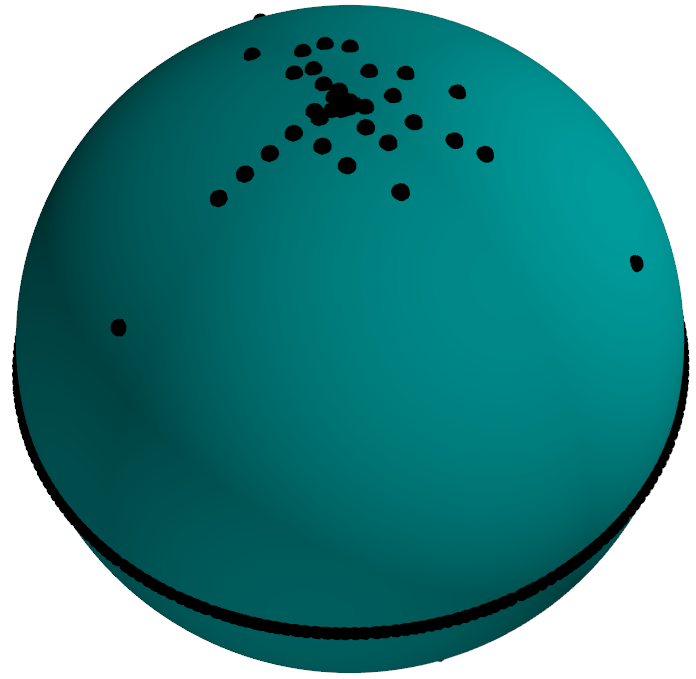}
\caption{$\mu_{.82}$}
\end{subfigure}
\caption{The results of the numerical experiments are given by $\mu_\lambda$}
\label{fig:experimentvisuals}
\end{figure}
\par
At $\lambda = .82$, this produced a point distribution that resembles the pole-equator measure, although some of the points are not exactly at the pole. We see a similar result for $\lambda = .78$, this time with more mass spread out between the pole and equator. Between $\lambda = .72$ and $\lambda = .78$, the experiment produces a 6-sided structure similar to a cube that is like neither $\sigma$ nor a pole-equator measure. Finally, at $\lambda = .72$ and below, the results  are spread out over $\mathbb S^2$  (hence, they somewhat resemble $\sigma$) although the distribution is visibly non-uniform. 
\par

We can then plot the $I_\lambda$ values for each of these measures. The results are shown in Figure \ref{fig:IalphaPlot}. Interestingly, we see that the results of our experiment do not always seem to produce the best optimizers. For example, $I_{.78}(\mu_{.82}) > I_{.78}(\mu_{.78})$ . So it seems likely that in this range, there are many measures which are local maximizers, but do not turn out to be global maximizers, which is quite common in energy optimization.
\begin{figure}[h]
\includegraphics[width=.7\textwidth]{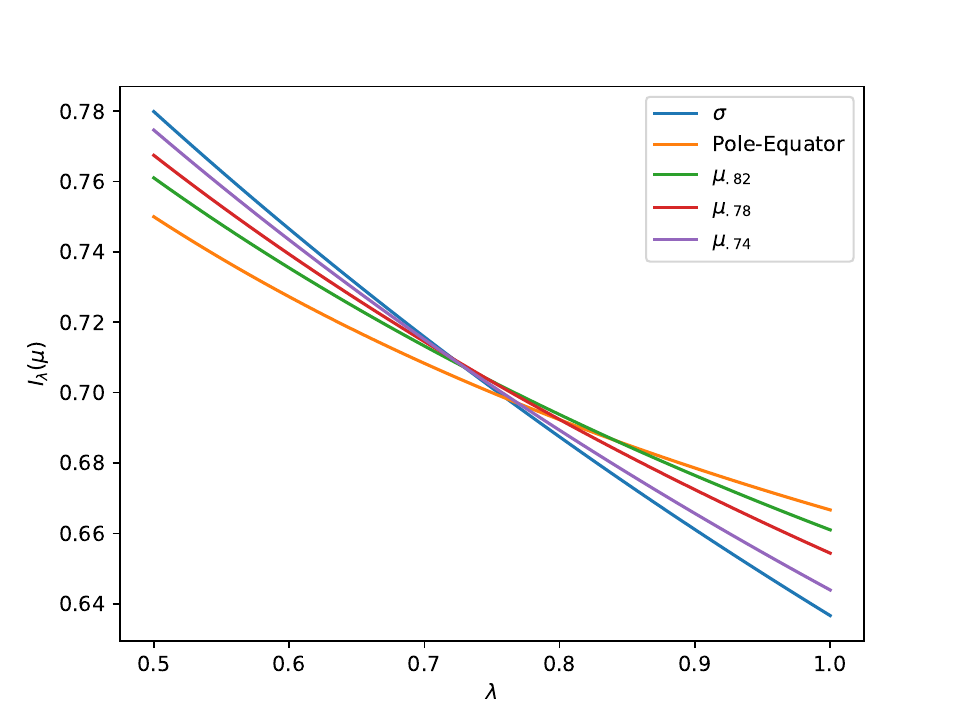}
\caption{$I_\lambda(\mu)$ for various measures}
\label{fig:IalphaPlot}
\end{figure}
\par
In conclusion, there is a wide variety of possible maximizers for $0.6 < \lambda \leq 1$ on $\mathbb{S}^2$. In particular, the range $[0.72,0.82]$ seems to have many local maxima, and is a strong candidate for a phase transition between measures similar to $\sigma$ and measure similar to pole-equator measures.

Finally, we want to note that numerical computations do appear to indicate that the Gegenbauer coefficients of $f_\lambda (t) = ( \arccos |t| )^\lambda$ are  negative when $0<\lambda<0.59$, which would imply  that $\sigma$ maximizes the energy $I_\lambda$ on $\mathbb S^2$ in this range, indicating that the range $-d < \lambda \le -(d-2)$ in Theorem \ref{t.main} is probably not sharp.


\section{Appendix}

In the Appendix we prove some technical statements, which at first glance might seem almost obvious, but upon more careful examination require proof. Some of these statements are much easier on the sphere, but less evident on more general spaces. This is mostly due to the abundance of literature about analysis on the sphere (in particular, spherical harmonics), and the scarcity of analogous literature on projective spaces. 

As before, let $\Omega$ be one of the compact connected two-point homogeneous spaces, i.e. $\mathbb{RP}^d$, $\mathbb{CP}^d$, $\mathbb{HP}^d$ ($d\ge 1 $), or $ \mathbb{OP}^d$ with $d=1$ or $2$.   We begin  by recalling the so-called {\emph{addition formula}}, see e.g. \cite{AndDGMS}. For $k\in \mathbb N$, let $Y_{k,1},\dots, Y_{k,m_k}$ denote real-valued functions on $\Omega$ which form an orthonormal basis of the eigenspace $V_k$ of the Laplace operator on $\Omega$ corresponding to the $k^{th}$ eigenvalue, where $m_k = \dim (V_k)$. Then for each $x, y \in \Omega$
 \begin{equation}\label{eq:addition}
 \sum_{m=1}^{m_k} Y_{k,m} (x) Y_{k,m} (y) = \frac{m_k}{P_k^{(\alpha,\beta)} (1) }  P_k^{(\alpha,\beta)} \big(\cos \vartheta (x,y) \big).
 \end{equation}
 
 We provide an immediate corollary to the addition formula.
 
 \begin{lemma}\label{l.convo}
 Let $k\in N$. Then for every $x,y \in \Omega$
 \begin{equation}\label{eq.convo}
 \int_{\Omega}  P_k^{(\alpha,\beta)} \big(\cos \vartheta (x,z) \big)  P_k^{(\alpha,\beta)} \big(\cos \vartheta (z,y) \big) d\sigma (z ) =  \frac{P_k^{(\alpha,\beta)} (1) }{m_k} \, P_k^{(\alpha,\beta)} \big(\cos \vartheta (x,y) \big). 
 \end{equation}
 \end{lemma}

   \begin{proof}
   We apply the addition formula \eqref{eq:addition} twice inside the integral in the left hand side of \eqref{eq.convo}, make use of orthonormality, and apply \eqref{eq:addition} once again. 
   \begin{align*}
    \int_{\Omega}  P_k^{(\alpha,\beta)} \big(\cos \vartheta (x,z) \big) &  P_k^{(\alpha,\beta)} \big(\cos \vartheta (z,y) \big) d\sigma (z ) \\ & =  \left( \frac{P_k^{(\alpha,\beta)} (1) }{m_k}  \right)^2 \int_\Omega \sum_{i,j =1}^{m_k} Y_{k,i} (x) Y_{k,i} (z) Y_{k,j} (z) Y_{k,j} (y)  d\sigma (z)\\
    &   =  \left( \frac{P_k^{(\alpha,\beta)} (1) }{m_k}  \right)^2 \,\,  \sum_{i =1}^{m_k} Y_{k,i} (x)  Y_{k,i} (y)     \int_\Omega Y^2_{k,i} (z)  d\sigma (z)\\
    & =  \frac{P_k^{(\alpha,\beta)} (1) }{m_k}  P_k^{(\alpha,\beta)} \big(\cos \vartheta (x,y) \big).
   \end{align*}
   \end{proof}
   On the sphere relation \eqref{eq.convo} could be immediately deduced from the Funk--Hecke formula, however, we were not able to find its analogue for projective spaces stated in the literature (although it would be a straightforward generalization). \,\,\,

 We now turn to the proof of the useful technical statements.  The first lemma shows that the condition $ I_{P_k^{(\alpha,\beta)} } (\mu ) =  0$ for all $k\ge 1$ uniquely defines the uniform measure $\sigma$ among all Borel probability measures on $\Omega$. This statement has been used several times, e.g., in the proofs of Proposition \ref{prop:sigma min abs cont measures bound density} and  Theorem  \ref{thm:sigma a minimizer for epsilon convergence}.

 \begin{lemma}\label{l.App1}
 Let $\mu$ be a Borel probability measure on $\Omega$.  Assume that for each $k \in \mathbb N $ we have $$ I_{P_k^{(\alpha,\beta)} } (\mu ) =  I_{P_k^{(\alpha,\beta)} } (\sigma ) = 0. $$
 Then $\mu = \sigma$. 
 \end{lemma}
   
 \begin{proof}
 Denoting the constant in \eqref{eq:addition} by $\gamma_k$ and using \eqref{eq.convo}, we obtain for all $k\ge 1$
 \begin{align*}
0 =  I_{P_k^{(\alpha,\beta)} }  (\mu ) & = \gamma_k  \int_\Omega \int_\Omega \int_\Omega  P_k^{(\alpha,\beta)} \big(\cos \vartheta (x,z) \big)  P_k^{(\alpha,\beta)} \big(\cos \vartheta (z,y) \big) d\sigma (z ) d\mu (x) d\mu (y)\\
 & =  \gamma_k  \int_\Omega \left(   \int_\Omega  P_k^{(\alpha,\beta)} \big(\cos \vartheta (x,z) \big)   d\mu (x)  \right)^2  \, d\sigma(z).
 \end{align*}
 Therefore, for all $k\ge 1$ we have $\int_\Omega  P_k^{(\alpha,\beta)} \big(\cos \vartheta (x,z) \big)   d\mu (x) = 0$ for all $z\in \Omega$ by continuity.  This implies that the relation 
$$ \int_\Omega  f  \big(\cos \vartheta (x,z) \big)   d\mu (x) = \int_\Omega   f \big(\cos \vartheta (x,z) \big)   d\sigma (x) $$
holds for every polynomial $f$, and by density, also for every continuous function $f \in C[-1,1]$. A further approximation shows that for each $a\in (-1,1) $, setting $r = \arccos a$,  we have that
$$ \mu \big( B_r (z) \big) =  \int_\Omega  \mathbf{1}_{(a,1]} (\cos \vartheta (x,z) \big)   d\mu (x) = \int_\Omega   \mathbf{1}_{(a,1]} \big(\cos \vartheta (x,z) \big)   d\sigma (x)  = \sigma \big( B_r (z) \big), $$ where $B_r (z)$ denotes a ball in $\Omega$ of geodesic radius $r$ with center at $z$,
i.e. the measures $\mu$ and $\sigma$ coincide on all balls in $\Omega$.

We note that in general metric spaces this is not sufficient to conclude that the two measures coincide, however, under some additional assumptions (e.g. if there exists a measure such that all balls of the same radius have the same measure, or if the space is finite dimensional in the sense that every ball of radius $r$ can be covered by a fixed number of balls of radius $r/2$, uniformly in $r>0$) the conclusion indeed holds, see e.g. \cite{Chr,Hof,Dav}. Since these assumptions clearly apply to $\Omega$, this implies that $\mu =\sigma$, which finishes the proof. 
 \end{proof}
   
   We now introduce a second technical lemma, which is an important first step in  the proof of Theorem \ref{thm:sigma a minimizer for epsilon convergence}. In this lemma, $\big\{ f^{(\varepsilon )} \big\}_{\varepsilon >0}$ will denote an arbitrary family of functions, not necessarily the approximations introduced in  \eqref{eq:fepsilon}.  We shall demonstrate that if $f^{(\varepsilon )} $ approximates $f$ in the sense of energy, it also approximates its Jacobi coefficients. 
   
   \begin{lemma}\label{l.App2}
   Assume that the functions $f^{(\varepsilon )} $ for $\varepsilon >0$ and  $f$ are in $L^1([-1,1], \nu^{(\alpha, \beta)})$  and are bounded from below. Assume also that for all $\mu \in \mathcal{P}(\Omega)$ such that $I_f(\mu) < \infty$,
\begin{equation*}
\lim_{\varepsilon \rightarrow 0^+} I_{f^{(\varepsilon)}}(\mu) = I_f(\mu).
\end{equation*}
Then for every $n\ge 0$, we have the convergence of Jacobi coefficients  
\begin{equation*}
\lim_{\varepsilon \rightarrow 0^+}  \widehat{f^{(\varepsilon)}}_n =  \widehat{f}_n.
\end{equation*}
\end{lemma}

\begin{proof}
First, we notice that the conclusion is trivial for $n=0$, since $I_f (\sigma) = \widehat{f}_0$ and the fact that $f \in L^1([-1,1], \nu^{(\alpha, \beta)})$  implies that $I_f (\sigma ) <\infty$. 

Now assume that $n\ge 1$. Using \eqref{eq.convo} and the fact that $\int_\Omega  P_n^{(\alpha,\beta)} \big(\cos \vartheta (x,z) \big)  d\sigma (z) = 0$, we find that 
\begin{equation*}
1+ c'_n  P_n^{(\alpha,\beta)} \big(\cos \vartheta (x,y) \big)  = \int_\Omega \left(1 + c_n P_n^{(\alpha,\beta)} \big(\cos \vartheta (x,z) \big) \right) \left( 1+ c_n P_n^{(\alpha,\beta)} \big(\cos \vartheta (z,y) \big)\right) \, d\sigma (z), 
\end{equation*}
where as before $c'_n = c^2_n m_n / P_n^{(\alpha,\beta)}(1)$ and $c_n$ is chosen in such a way that $1 + c_n P_n^{(\alpha,\beta)}  \ge 0$. Therefore we can write
\begin{align*}
\widehat{f^{(\varepsilon )}}_0 +  c'_n  \widehat{f^{(\varepsilon )}}_n &  = \int_\Omega \int_\Omega f^{(\varepsilon )}  \big(\cos \vartheta (x,y) \big)  \left( 1+ c'_n  P_n^{(\alpha,\beta)} \big(\cos \vartheta (x,y) \right) \, d\sigma(x) d\sigma(y)\\  
& = \int_\Omega  I_{ f^{(\varepsilon )} }  (\mu_z) \, d\sigma (z) = I_{ f^{(\varepsilon )} }  (\mu_{z_0}) \,\,\, \textup{ for any }\,\, z_0 \in \Omega,
\end{align*}
where we set $d\mu_z (x)  = \left(1 + c_n P_n^{(\alpha,\beta)} \big(\cos \vartheta (x,z) \big) \right) d\sigma (x)  $ and observe that
\begin{equation*}
 I_{ f^{(\varepsilon )} }  (\mu_{z}) = \int_\Omega \int_\Omega  f^{(\varepsilon )} \big(\cos \vartheta (x,y) \big) \,  \left(1 + c_n P_n^{(\alpha,\beta)} \big(\cos \vartheta (x,z) \big) \right) \left( 1+ c_n P_n^{(\alpha,\beta)} \big(\cos \vartheta (z,y) \big)\right) d\sigma (x) d\sigma (y) 
\end{equation*}
is independent of $z\in \Omega$ due to two-point homogeneity of $\Omega$ and isometry invariance of $\sigma$. 

Performing an identical computation for $f$ we observe that $I_f (\mu_z ) <\infty$, since Jacobi coefficients of $f$ are well-defined. Therefore, by the assumption of the lemma,  $I_{ f^{(\varepsilon )} }  (\mu_z)  \rightarrow I_{ f }  (\mu_z) $ as $\varepsilon \rightarrow 0$, and thus 
\begin{equation*}
\lim_{\varepsilon \rightarrow 0^+} \left( \widehat{f^{(\varepsilon )}}_0 +  c'_n  \widehat{f^{(\varepsilon )}}_n  \right) = \widehat{f}_0 +  c'_n  \widehat{f}_n. 
\end{equation*}
Since the lemma is already established for $n=0$, this finishes the proof. 
\end{proof}

\section{Acknowledgments}
The authors would like to thank Carlos Beltr\'{a}n, Martin Ehler, Peter Grabner, Doug Hardin, Danylo Radchenko, Ed Saff, and Stefan Steinerberger for their insights. The first author was funded by the NSF grant DMS 2054606. The second author was funded by the Austrian Science Fund FWF project F5503 part of the Special Research Program (SFB) “Quasi-Monte Carlo Methods: Theory and Applications,” and the NSF Mathematical Sciences Postdoctoral Research Fellowship Grant 2202877.

\end{document}